\documentclass[12pt]{article}
\usepackage{url}
\usepackage{mathrsfs}
\usepackage[a4 paper, margin=2cm]{geometry}

\usepackage{enumerate}
\usepackage{graphics}
\usepackage{tikz}

\usepackage{verbatim}
\usepackage{slashed}
\usepackage[hidelinks]{hyperref}

\usepackage{graphicx}
\usepackage{datetime}
\usepackage{color}
\usepackage{amsmath,amsfonts,amsthm,amssymb}
\usepackage{epstopdf}

\newtheorem{theorem}{Theorem}[section]
\newtheorem{proposition}[theorem]{Proposition}

\newtheorem{corollary}[theorem]{Corollary}
\newtheorem{lemma}[theorem]{Lemma}
\newtheorem{remark}[theorem]{Remark}
\numberwithin{equation}{section}

\newcommand{\R}{\mathbb{R}}
\newcommand{\N}{\mathbb{N}}
\newcommand{\cN}{\mathcal{N}}
\newcommand{\cQ}{\mathcal{Q}}
\newcommand{\cK}{\mathcal{K}}

\newcommand{\HOJ}{\mathbb{H}}
\newcommand{\HO}{\mathbb{H} }
\newcommand{\cE}{\mathcal{E}}
\newcommand{\HOOJ}{\mathbb{H}_{0}}
\newcommand{\HOO}{\mathbb{H}_{0}}
\newcommand{\D}{\mathfrak{D} }
\newcommand{\nN}{\mathfrak{N} }
\DeclareMathOperator{\dist}{dist}
\newcommand{\e}{\varepsilon}

\begin{document}

\title{A logistic equation with non-local operators \\
of order near zero}
\author{Luigi Appolloni, Serena Dipierro and Enrico Valdinoci\thanks{Luigi Appolloni: University of Leeds, School of Mathematics, Woodhouse, Leeds LS2 9JT, United Kingdom\\
{\tt L.Appolloni@leeds.ac.uk}
\\
Serena Dipierro: Department of Mathematics and Statistics, The University of Western Australia, 35 Stirling Highway, Crawley, Perth, WA 6009, Australia\\
{\tt serena.dipierro@uwa.edu.au}
\\
Enrico Valdinoci: Department of Mathematics and Statistics, The University of Western Australia, 35 Stirling Highway, Crawley, Perth, WA 6009, Australia\\
{\tt enrico.valdinoci@uwa.edu.au}
}}
\maketitle

\begin{abstract}
In this paper we study an equation driven by a non-local operator of order ``near~zero'' with mixed Dirichlet--Neumann boundary conditions, modelling a biological system consisting of a population self-competing for resources in a closed environment and subject to an extremely heavy-tailed dispersal process. We establish several results concerning the existence of states in which the population can survive requiring different assumptions on the non-linear reaction term and on the various configurations of the sets where the Dirichlet or Neumann conditions hold.

\end{abstract}
\section{Introduction}
In this paper we investigate the existence of non-negative and non-trivial solutions to the following problem
\begin{equation}
\label{mainproblem}
\tag{$P$}
\begin{cases}
\mathfrak{L} u = f(x,u) & \text{in } \Omega, \\
\mathcal{N} u = 0 & \text{in } \nN, \\
u = 0 & \text{in } \mathfrak{D}.
\end{cases}
\end{equation}
Here, $\Omega \subset \mathbb{R}^N$ is a bounded domain with $N \geq 1$, while $\mathfrak{N}$ and $\mathfrak{D}$ are open subsets of $\mathbb{R}^N \setminus \Omega$ such that~$\mathfrak{N}\cap \mathfrak{D}=\varnothing$ and~$ \mathbb{R}^N \setminus \Omega =\overline{\mathfrak{N}\cup \mathfrak{D}}$. We assume that the boundaries of~$\Omega$,
$\mathfrak{N}$ and~$\mathfrak{D}$
are smooth. The symbol $\mathfrak{L}$ denotes a non-local operator, $\mathcal{N}$ is the corresponding Neumann-type operator, and $f$ represents a non-linear reaction term.

The prototype of nonlinearity that we have in mind is the logistic one, namely
\begin{equation} \label{eq34}
    f(x,t) := \left(m(x) - \mu(x)\, t\right) t,
\end{equation}
where $m$, $ \mu \colon \Omega \to \R$ are two measurable functions representing, respectively, the available resources and the competition for those resources in the region $\Omega$.

In this setting, problem \eqref{mainproblem} can be interpreted as a model for a biological system in which~$u$ represents the density of a population self-competing for resources within the region $\Omega$. In this setting, a non-negative and non-trivial solution corresponds to a configuration in which the species can survive. The population is allowed to spread and self-compete for the resources in the environment according to a non-local diffusion process driven by the operator $\mathfrak{L}$.
The request $u = 0$ in $\D$ is a non-local Dirichlet-type condition and therefore $\D$ represents a region where the population cannot survive. In contrast, individuals are allowed to interact between the regions $\Omega$ and $\nN$, and the condition $\mathcal{N} u = 0$ in $\nN$ imposes a balance of the exchanges between the two sets. 

Using variational methods and spectral analysis techniques, in this paper we show that problem \eqref{mainproblem} admits non-negative and non-trivial solutions under different configurations of the sets $\Omega, \  \nN, \ \D$ and for a general class of operators that we will call of order ``near zero". More precisely, the operator $\mathfrak{L}$  is defined as 
\begin{equation} \label{eq33}
\mathfrak{L}u(x) := \mathrm{P.V.} \int_{\mathbb{R}^N} \bigl(u(x) - u(y)\bigr)\mathcal{K}(x-y)\, dy,
\end{equation} where~$P.V.$ stands for the Cauchy principal value notation.

The operator $\mathcal{N}$ defines the corresponding non-local Neumann condition which, for $x \in \mathbb{R}^N \setminus \overline{\Omega}$, is given by
\begin{equation*}
\mathcal{N}u(x):
=
\int_{\Omega} (u(x)-u(y))\mathcal{K}(x-y)\,dy.
\end{equation*}

Regarding the kernel~$\mathcal{K}$, we assume that
$\mathcal{K} \colon \mathbb{R}^{N} \setminus\{0\}\to [0,+\infty)$ is a measurable function satisfying the following conditions:
\begin{itemize}
\item[\textbf{\((K_1)\)}]  $\mathcal{K}(z) = \mathcal{K}(-z)$ and, for all~$r>0$,
\begin{equation*}
\int_{\R^N\setminus B_r} \mathcal{K}(z) \, dz < +\infty.
\end{equation*}

\item[\textbf{\((K_2)\)}] There exist $\rho >0$ and a measurable function~$\ell \colon (0, \rho) \to (0, + \infty)$ such that
\begin{equation*}
\mathcal{K}(z) =\frac{ \ell(|z|)}{ |z|^{N}} \quad \text{for } 0 < |z| < \rho.
\end{equation*}
    
\item[\textbf{\((K_3)\)}] The kernel $\mathcal{K}$ belongs to the non-integrable regime, i.e.
\begin{equation*}
\lim_{t\to0^+} \int_t^\rho \frac{\ell(\vartheta)}{\vartheta} \, d \vartheta =+ \infty.  
\end{equation*}

\item[\textbf{\((K_4)\)}] The function $\ell$ varies slowly at the origin, i.e.
\begin{equation*}
\lim_{t \to 0} \frac{\ell(\alpha t)}{\ell(t)} = 1 \quad \text{for every } \alpha > 0.
\end{equation*}
\end{itemize}
    
We point out that the case of the fractional Laplacian operator,
which corresponds to the choice
\begin{equation*}
\mathcal{K}(z) := |z|^{-N - 2s}, \quad {\mbox{with }} 0 < s < 1,
\end{equation*}
is not comprised in our setting, as this kernel would not satisfy~$(K_4)$.

Instead, in the present paper, we focus on weakly singular kernels, representing a {\em borderline regime between the fractional Laplacian and the integrable case}, with the aim of modeling
diffusive processes with {\em
extremely heavy-tailed dispersal}. 

Functions $\ell$ satisfying assumption $(K_4)$ are commonly referred to in the literature as ``slowly varying''. Typical examples include $\ell(t) = 1$, $\ell(t) = \log^\beta(2\rho/t)$ for $\beta \ge -1$, and $\ell(t) = \left( \log(2\rho/t)\log\log(2\rho/t) \right)^{-1}$. On the other hand, rapidly oscillating functions such as $\ell(t) = 1 - \sin(1/t)$ are excluded from our analysis.

We emphasize that similar settings have already been considered in other papers (see \cite{MR3759570} and the references therein). Moreover, we also mention \cite{MR4745743, MR4609798, MR4741033, MR3995092, MR4795535, MR4400612, MR4141492, MR4279386, MR4932276},  where other problems regarding operators of order ``near zero" are investigated. 

Problems of logistic type driven by non-local operators have been extensively studied in recent years. For instance, we mention \cite{MR3579567}, where the authors studied the Dirichlet case for a logistc equation driven by the fractional Laplacian, and \cite{MR4651677}, where the Neumann case was considered for a mixed-type local/non-local operator. See also~\cite{MR5010423}
and the references therein for more comments on logistic-type
equations driven by non-local operators.

Of particular relevance to the present work is \cite{MR3771424}, where the authors investigated a related problem involving the classical fractional Laplacian, with particular emphasis on the role of the fractional exponent $s$. In their framework, the parameter $s$ describes how ``static'' the population is. They showed that survival occurs when the fractional exponent is either $s = 1$ or $s \to 0^+$, corresponding to regimes where the diffusion is stronger or long-range interactions are less penalized, respectively. 
Additionally, in the analysis of optimal animal foraging,
it has been shown in~\cite{MR4645663, MR4711702}
that if the prey is close enough to the forager, then the most rewarding hunting strategy lies in a small neighborhood of~$s=0$,
notwithstanding the fact that~$s=0$ is a global pessimizer for some of the efficiency functionals. 

In terms of biological data, extreme heavy-tailed kernels
have been detected in several experiments, though 
empirical tracking data cannot precisely capture
the exact value of an interaction kernel in a critical regime, due to
inherent spatial error margins of the device used (such as
GPS telemetry, radar, camera traps). However,
it has been suggested in~\cite{10.1098rspb.2013.2997} that these kinds of patterns
are related to ``ambush'' strategies by predators in the presence
of high-energy content preys.

These results motivate the analysis of the limiting regimes for operators with weakly singular kernels, which we address here in the case of mixed boundary conditions.

Concerning the non-linear term in \eqref{mainproblem}, we assume that $f: \Omega \times \R \to \R$ is a Carathéodory function such that $f(x,0)=0$, and we define 
\begin{equation*}
    F(x,t):=\int_0^t f(x,\theta)\, d\theta.
\end{equation*}
We require the following assumptions:
\begin{itemize}
     \item[$(f_1)$] There exist $L>0$ and $A \in L^1(\Omega)$, such that
     \begin{equation*}
\begin{cases}
|F(x,t)| \le A(x) & \text{for a.e. } x \in \Omega \text{ and all } t \in [0,L],\\
F(x,t) \le F(x,L) & \text{for a.e. } x \in \Omega \text{ and all } t > L.
\end{cases}
\end{equation*}
    \item[$(f_2)$] $F(x,|t|) \geq F(x,t)$ for a.e. $x\in\Omega$ and all $t \in \R$.
\end{itemize}

The strategy for finding solutions to \eqref{mainproblem} is to look for critical points of the functional 
\begin{equation*}
    \cE(u):= \frac{1}{4}\iint_{\cQ} \left(u(x)-u(y)\right)^2 \mathcal{K} (x-y) \,dx dy - \int_\Omega F(x,u) \, dx,
\end{equation*}
where $\cQ := \R^{2N} \setminus (\Omega^c \times \Omega^c)$. Assumptions $(f_1)$--$(f_2)$ ensure that the second term in the energy functional is well defined after performing a truncation. Moreover, from a variational point of view, these assumptions imply that truncations and taking absolute values do not increase the energy.

We are now ready to state our first result, which shows that, under the assumptions on the kernel and the nonlinearity introduced above, the functional $\cE$ attains a minimum.

\begin{theorem} \label{th1} 
Suppose that $\mathcal{K} \colon \mathbb{R}^{N} \setminus\{0\}\to [0,+\infty)$ is a measurable function satisfying $(K_1)$--$(K_4)$ and $f \colon \Omega \times \R \to \R$ satisfies $(f_1)$--$(f_2)$. Then problem \eqref{mainproblem} admits a bounded, non-negative solution, which is obtained as a minimizer of $\cE$.
\end{theorem}

We point out that, in the case of the logistic equation~\eqref{eq34}, 
we take
$$ F(x,t):=\frac{m(x)t^2}2-\frac{\mu(x)|t|^3}3,$$
which, under suitable assumptions on~$m$ and~$\mu$, satisfies~$(f_1)$--$(f_2)$.
Therefore, Theorem~\ref{th1} produces a non-negative minimizer
of the energy, and as a result a non-negative solution to~\eqref{mainproblem} with~$f$ as in~\eqref{eq34}.

The minimization in Theorem \ref{th1} is carried out in the space $\HOJ(\Omega \cup \nN)$, which is defined as
\begin{equation*}
\begin{aligned}
\HOJ(\Omega \cup \nN) := \bigg\{& u : \R^N \to \R \ \text{measurable} \ : \ 
 u|_{\Omega \cup \nN} \in L^2(\Omega \cup \nN) \\
& \quad\text{and } \iint_{\widetilde{\cQ}} \left(u(x) - u(y)\right)^2 \mathcal{K}(x- y) \, dx \, dy <+ \infty 
\bigg\},
\end{aligned}
\end{equation*}
where $\widetilde{\cQ} := \R^{2N} \setminus \big((\Omega \cup \nN)^c \times (\Omega \cup \nN)^c\big)=
\R^{2N} \setminus (\mathfrak{D}\times \mathfrak{D})$.

Under assumptions $(K_1)$--$(K_4)$, it is possible to show that $\HOJ(\Omega \cup \nN)$ is a Hilbert space with inner product
\begin{equation*}
\langle u, w \rangle := \langle u, w \rangle_{L^2(\Omega \cup \nN)} + \iint_{\widetilde{\cQ}} \left(u(x) - u(y)\right) \left(w(x) - w(y)\right) \mathcal{K}(x- y) \, dx \, dy
\end{equation*}
and induced norm
\begin{equation*}
\| u \| = \left(\| u \|_{L^2(\Omega \cup \nN)}^2 + 
\iint_{\widetilde{\cQ}} \left(u(x) - u(y)\right)^2\mathcal{K}(x- y) \, dx \, dy\right)^{1/2}.
\end{equation*}

For future reference, we also introduce the smaller space
\begin{equation*}
\HOOJ (\Omega \cup \nN) := \left\{ u \in \HOJ(\Omega \cup \nN) : u = 0 \ \text{in} \ (\Omega \cup \nN)^c =\mathfrak{D}\right\}.
\end{equation*}

\begin{remark} \label{compactness}{\rm We point out that,
even if the set $\Omega \cup \nN$ is bounded, the space $\HO(\Omega \cup \nN)$ is not known to embed compactly into $L^2(\Omega \cup \nN)$ (or other Lebesgue spaces in general). To the best of our knowledge, no counterexample showing that the embedding $\HOJ(\Omega \cup \nN) \hookrightarrow L^2(\Omega \cup \nN)$ fails to be compact is currently available in the literature. 

So far, some compactness results in this direction have been obtained 
in \cite{MR3759570} only under more restrictive assumptions on the function $\ell$ appearing in hypothesis~$(K_2)$.

On the other hand, if one works in the smaller space $\HOO(\Omega \cup \nN)$, then the embedding $\HOOJ(\Omega \cup \nN) \hookrightarrow L^2(\Omega \cup \nN)$ is compact whenever $\Omega \cup \nN$ is bounded.}
\end{remark}

This possible lack of compactness is one of the main challenges that
we need to address, as it prevents the straightforward application of standard variational methods.

In order to overcome this problem, in this paper we introduce a new strategy which, roughly speaking, is as follows: we take a minimizing sequence of $\cE$ in the space $\HOOJ(\Omega \cup \nN)$ for the problem with the Neumann condition. Then, we modify this sequence in the region $\nN$ in order to work in the smaller space $\HOOJ(\Omega)$ and exploit its compactness properties. After extracting a convergent subsequence, we show that compactness can be recovered also for the original sequence and that a minimum can be suitably defined also in the region $\nN$. The crucial point of this strategy is to carry out this procedure without losing control of the boundedness of the energy $\cE$, and consequently of the Sobolev seminorm, in order to be able to extract subsequences.

We now focus on identifying conditions that determine whether the solution found in Theorem~\ref{th1}
is non-trivial. For this, we consider the case in which the nonlinearity has the specific logistc form given in \eqref{eq34}, which leads us to studying the problem
\begin{equation}
\label{mainlogistic}
\tag{$P_1$}
\begin{cases}  
\mathfrak{L} u = (m(x)-\mu(x)u)u  & \mbox{in } \Omega, \\
\mathcal{N} u(x) = 0 & \mbox{in } \mathfrak{N}, \\
u = 0 & \mbox{in } \mathfrak{D}.
\end{cases}
\end{equation}

The function $m$ represents the availability of resources in the region $\Omega$ and may take both positive and negative values. Regions where $m > 0$ indicate favourable conditions that support reproduction, while areas where $m < 0$ correspond to a hostile environment, leading to a linearly increasing mortality rate. The population, described by the function $u$, self-competes for the resources available in the environment, and the competition is modelled by the positive function $\mu$.

As intuition may suggest, the absence of resources in the region $\Omega$ implies that the environment is not favourable for survival and our first result goes in this direction. Namely, we prove that when $m$ is non‑positive, then the species cannot survive and the minimum is attained by the trivial function that is zero in the whole space. On the other hand, if the integral of $m$ is sufficiently large and the domain $\Omega$ is surrounded by the region $\mathfrak{N}$, then we are able to prove that the species survives. More precisely:

\begin{theorem}\label{th2}
Suppose that $\mathcal{K}$ satisfies $(K_1)$--$(K_4)$, $m\in L^\infty(\Omega)$ and $\mu\in L^1(\Omega)$. Assume there exists a constant $\overline{\mu}>0$ such that $\mu(x)\geq \overline{\mu}$ for a.e.~$x\in\Omega$. 

Then the following statements hold true:
\begin{enumerate}
    \item[(i)] If $m$ is non-positive, then the only non-negative solution to \eqref{mainlogistic} is the trivial one.
    \item[(ii)] 
Assume that
    \begin{equation} \label{eq38}
    \int_\Omega m(x)\,dx > \iint_{\Omega\times\Omega^c} \mathcal{K}(x-y)\,dx\,dy
    \end{equation}
and that
    \begin{equation} \label{domain}
    \overline{\Omega} \cap\overline{\D}
    =\varnothing.
\end{equation}

Then problem \eqref{mainlogistic} admits a non-negative and non-trivial solution.
\end{enumerate}
\end{theorem}

We point out that the integral on the right-hand side  of~\eqref{eq38} is finite, see Corollary~\ref{OJSL0oLK}.

The strategy of Theorem \ref{th2} for obtaining the non-trivial minimizer is to construct a suitable
competitor with negative energy. To do this,
the assumption~\eqref{domain} comes into play,
guaranteeing that the region $\nN$ completely surrounds the domain $\Omega$, and therefore the hostile region~$\D$ is not
adjacent to~$\Omega$ (this condition has a clear biological interpretation, as it asserts that survival occurs when the extinction region does not affect directly
the ecological niche where the population is free to spread). The idea to construct the competitor is then
to take it constant on the region $\Omega$ and let it decay to zero within the region $\nN$ (this will be possible thanks to assumption~\eqref{domain}). 

In line with the existing literature, see e.g.~\cite{MR3579567, MR4651677}, one can relate the problem of finding non-trivial
solutions with the study of certain eigenvalue problems.
The analysis of the spectral properties
of the linearized problem will be carried out in Section~\ref{section4}
and is useful to establish existence and non-existence
of trivial solutions for the logistic equation. For this analysis,
we work with the specific kernel
\begin{equation}\label{cbnweiyft3y}
\mathcal{K}(z) := |z|^{-N} \chi_{B_1}(z).
\end{equation}
and we see that:

\begin{theorem} \label{th4}
Let~$\mathcal{K}$ be as in~\eqref{cbnweiyft3y},
$m \in L^\infty(\Omega)$ and~$\mu\in L^1(\Omega)$.
Assume there exists a constant $\overline{\mu}>0$ such that $\mu(x)\geq \overline{\mu}$ for a.e.~$x\in\Omega$.

Suppose also that $\nN$ is bounded and that there exists $d_0 > 0$ such that $\dist(\Omega, \nN) \geq d_0$.

Then the following statements hold true: 
\begin{description}
    \item[$(i)$] There exists $M_1>0$ such that if $\sup_{\Omega} m \geq M_1$ and there exists an open set~$\mathcal{S} \subset \Omega$ such that $m>0$ a.e. in $\mathcal{S}$, then problem \eqref{mainlogistic} admits a non-negative and non-trivial solution.
    \item[$(ii)$] There exists $M_2>0$ such that\footnote{As customary, we use the following notation for the positive and
    negative part of a function:
    $$ f^+(x):=\max\{0,f(x)\}\qquad{\mbox{and}}\qquad
    f^-(x):=\max\{0,-f(x)\}
    .$$} if $\Vert m^+ \Vert_{L^\infty(\Omega)} < M_2$, then the only solution to~\eqref{mainlogistic} is the trivial one.
\end{description}
\end{theorem}

We point out that, even when requiring $\nN$ to be bounded, problem \eqref{mainlogistic} remains quite challenging. Indeed, despite working in a Sobolev space with $u=0$ in $\D$ helps to recover compactness, there is no direct correspondence between the seminorm of the Sobolev space we are working with and the seminorm appearing in the functional. This will require to make some careful estimates in order to show that minimizing sequences are bounded and to exploit the compactness properties of the space $\HOO(\Omega \cup \nN)$.

As mentioned above, the strategy of the proof of Theorem \ref{th4} relies on the study of two weighted eigenvalue problems
(as a matter of fact, the first eigenfunction of a suitable linear problem will be used as a competitor). The hypothesis $\dist(\Omega, \nN) \geq d_0$ will be crucial when carefully estimating the Gagliardo seminorm after a normalization of the minimizing sequence, which is necessary to satisfy the non-local Neumann condition.
\medskip

The rest of the paper is organized as follows. In Section~\ref{rem:kappa} we make some remarks on the assumptions on the kernel~$\mathcal{K}$ and provide some
useful results.
In Section~\ref{section3} we prove Theorems~\ref{th1} and~\ref{th2}. Section~\ref{section4} contains the study of qualitative properties of two eigenvalue problems, which are necessary for the proof of Theorem~\ref{th4}, presented in Section~\ref{section5}.

\section{On the kernel~\texorpdfstring{$\mathcal{K}$}{K}}\label{rem:kappa}

In this section we make some explicit estimates on the 
slowly varying function~$\ell$ that will help estimate integral contributions involving the kernel~$\mathcal K$. 

\begin{lemma}\label{KARA}
For every~$\e\in(0,1]$
there exists~$\eta_\e\in(0,\rho)$ such that, for all~$k\in\N$,
\begin{equation}\label{4:9}
\sup_{[\eta_\e/2^k,\eta_\e]}\ell\le  
\frac{\ell(\eta_\e)}{\left(1-\e\right)^k}
\qquad{\mbox{and}}\qquad
\inf_{[\eta_\e/2^k,\eta_\e]}\ell\ge  \frac{\ell(\eta_\e)}{\left(1+\e\right)^k}.
\end{equation}
\end{lemma}

\begin{proof}
We use the measurability of~$\ell$ and Karamata’s Uniform Convergence Theorem (see e.g. Theorem~1.2.1 in~\cite{MR898871}) and we get that
\begin{equation*}
\lim_{t\to0^+}\sup_{\alpha\in[1/2,2]}\left|\frac{\ell(\alpha t)}{\ell(t)}-1\right|=0.
\end{equation*}
Hence, we find~$\eta_\e\in(0,\rho)$ such that, for all~$t\in[0,\eta_\e]$,
\begin{equation*}
\sup_{\alpha\in[1/2,2]}\left|\frac{\ell(\alpha t)}{\ell(t)}-1\right|\le\e\end{equation*}
and thus, for all~$t\in[0,\eta_\e]$ and~$\alpha\in\left[\frac12,2\right]$,
\begin{equation}\label{4:q}
\frac{\ell(\alpha t)}{\ell(t)}\in\left[1-\e,1+\e\right].\end{equation}

The proof of~\eqref{4:9} is now by induction. If~$k=0$, the result is obvious, so suppose that~\eqref{4:9} holds true for some index~$k$ and let us prove it for the index~$k+1$. To this end, let~$t\in\left[\frac{\eta_\e}{2^{k+1}},\frac{\eta_\e}{2^k}\right]$
and~$\alpha_t:=\frac{\eta_\e}{2^k t}\in\left[\frac12,2\right]$.

It follows from~\eqref{4:q} that
$$ [1-\e,1+\e]\ni
\frac{\ell(\alpha_t t)}{\ell(t)}
=\frac{\ell(\eta_\e/2^k)}{\ell(t)}.$$
Accordingly, using the inductive assumption,
$$\ell(t)\le\frac{\ell(\eta_\e/2^k)}{1-\e}\le \frac{\ell(\eta_\e)}{\left(1-\e\right)^{k+1}}$$
and
$$\ell(t)\ge\frac{\ell(\eta_\e/2^k)}{1+\e}\ge \frac{\ell(\eta_\e)}{\left(1+\e\right)^{k+1}},$$
from which one completes the inductive step and proves~\eqref{4:9}.
\end{proof}

\begin{corollary}\label{ORDER:Z}
For all~$\delta\in(0,1]$ there exist~$\theta_\delta\in(0,\rho)$ and~$C_\delta\in(1,+\infty)$ such that, for all~$t\in(0,\theta_\delta]$,
\begin{equation}\label{DF:S} \frac{t^\delta}{C_\delta}\le\ell(t)\le\frac{C_\delta}{t^\delta}.\end{equation}
\end{corollary}

\begin{proof} We use Lemma~\ref{KARA}
and, given~$t\in(0,\eta_\e]$, we pick~$k\in\N$ such that~$t\in
\left[\frac{\eta_\e}{2^{k+1}},\frac{\eta_\e}{2^k}\right]$, finding that
$$ \ell(t)\le\sup_{[\eta_e/2^{k+1},\eta_\e]}\ell\le  \frac{\ell(\eta_\e)}{\left(1-\e\right)^{k+1}}=\frac{
2^{k\log_2 \frac1{1-\e}}\ell(\eta_\e)}{1-\e}\le
\frac{\eta_\e^{\log_2\frac1{1-\e}}\ell(\eta_\e)}{t^{\log_2\frac{1}{1-\e}}(1-\e)},
$$
which gives the upper bound in~\eqref{DF:S}
by choosing~$\e:=1-2^{-\delta}$.

On a similar note,
$$\ell(t)\ge
\inf_{[\eta_\e/2^{k+1},\eta_\e]}\ell\ge \frac{\ell(\eta_\e)}{ \left(1+\e\right)^{k+1}}
=\frac{\ell(\eta_\e)}{2^{k\log_2(1+\e)}(1+\e)}\ge
\frac{t^{\log_2(1+\e)}\ell(\eta_\e)}{\eta_\e^{\log_2(1+\e)}(1+\e)}
,$$
which gives the lower bound in~\eqref{DF:S}
by choosing~$\e:=2^{\delta}-1$.
\end{proof}

We observe that Corollary~\ref{ORDER:Z}
showcases an interesting feature of the kernel~$\mathcal{K}$, which also confirms the idea that the corresponding operator
is of order near zero, due to the arbitrariness of~$\delta$
in Corollary~\ref{ORDER:Z}.

Notice also that Corollary~\ref{ORDER:Z} entails that, for every~$s\in(0,1)$,
\begin{equation*}
\lim_{z \to 0} |z|^{N + 2s} \mathcal{K}(z) = 
\lim_{z \to 0}  |z|^{2s}\ell(|z|)=0.
\end{equation*}
This entails that
the singularity at the origin is milder than in the fractional Laplacian setting.

\begin{lemma}\label{bfewERTYU3892gyf}
Given a bounded and smooth set~$U\subset \R^N$, for all~$c\in[0,1)$, we have that
\begin{equation*}
\iint_{{U\times U^c}\atop{\{|x-y|\le1\}}}\frac{dx\,dy}{|x-y|^{N+c}}<+\infty.
\end{equation*}
\end{lemma}

\begin{proof} For all~$x\in U$, let~$R_x:=\dist(x,\partial U)$.
Then,
\begin{eqnarray*}
&&\iint_{{U\times U^c}\atop{\{|x-y|\le1\}}}\frac{dx\,dy}{|x-y|^{N+c}}\leq
\iint_{U\times (B_1(x)\setminus B_{R_x}(x))}\frac{dx\,dy}{|x-y|^{N+c}}
\\&&\qquad=\iint_{U\times (B_1\setminus B_{R_x})}\frac{dx\,dz}{|z|^{N+c}}=
\omega_{n-1}\int_{U}\left(\int_{R_x}^1
\frac{dr}{r^{1+c}}\right)\,dx\\&&\qquad
=\begin{cases}\displaystyle
\frac{\omega_{n-1}}
c\int_{U}\left(\frac1{R_x^c}-1\right)\,dx&
{\mbox{ if }} c\in(0,1),\\ \displaystyle
-\omega_{n-1}\int_{U}\ln R_x\,dx
&{\mbox{ if }} c=0,\end{cases}
\end{eqnarray*}
where~$\omega_{n-1}$ is the surface area of the unit sphere in~$\R^N$.

Now, given $t >0$ sufficiently small, we define
\begin{equation*}
U_t := \{ x \in U : \operatorname{dist}(x, \partial U) < t \}
\end{equation*}
and observe that this is an open set with Lipschitz boundary. Consequently, for any~$\delta > 0$, we find~$t_*$ 
(depending on~$U$ and~$\delta$) such that
\begin{equation*}
|\partial U_t| \leq (1 + \delta) |\partial U| \quad \text{for all } t \in (0, t^*).
\end{equation*}

We first suppose~$c\in(0,1)$ and we use 
the coarea formula to see that
\begin{eqnarray*}
\iint_{{U\times U^c}\atop{\{|x-y|\le1\}}}\frac{dx\,dy}{|x-y|^{N+c}}& \leq&\frac{\omega_{n-1}}
c\int_{U}\frac{dx}{R_x^c}
\\
 &=&  \frac{\omega_{n-1}}{c} \int_{U_{t_*}} \frac{dx}{R_x^c}  
+ \frac{\omega_{n-1}}{c} \int_{U \setminus U_{t^*}} \frac{dx}{R_x^c}  \\
& \leq &\frac{\omega_{n-1}}{c} \int_0^{t_*} \frac{ |\partial U_t|}{t^c} \, dt +  \frac{\omega_{n-1} |U|}{c\,t_*^{c}} \\
& \leq &\frac{\omega_{n-1}}{c(1 - c)} (1 + \delta) |\partial U| t_*^{1 - c} + \frac{\omega_{n-1}|U|}{c\, t_*^{c} },
\end{eqnarray*} which is finite, since~$U$ is bounded.

If instead~$c=0$, we find that
\begin{eqnarray*}
\iint_{{U\times U^c}\atop{\{|x-y|\le1\}}}\frac{dx\,dy}{|x-y|^{N}}& =&-\omega_{n-1}\int_{U}\ln R_x\,dx
\\
 &=& -\omega_{n-1} \int_{U_{t_*}} \ln R_x\,dx  
- \omega_{n-1} \int_{U \setminus U_{t^*}} \ln R_x\,dx  \\
& \leq &-\omega_{n-1} \int_0^{t_*} |\partial U_t|\,\ln t \, dt -\omega_{n-1} |U|\ln t_* \\
& \leq &-\omega_{n-1}(1 + \delta) |\partial U|t_*(\ln t_*-1) - \omega_{n-1} |U|\ln t_*,
\end{eqnarray*}
which is finite.

These observations complete the proof of Lemma~\ref{bfewERTYU3892gyf}.\end{proof}

\begin{corollary}\label{coro:bfewERTYU3892gyf}
Let~$U\subset \R^N$ be a bounded and smooth set.
Then, there exists~$\theta\in(0,\rho)$ such that
\begin{equation*}
\iint_{{U\times U^c}\atop{\{|x-y|\le\theta\}}}\mathcal{K}(x-y)\,dx\,dy<+\infty.
\end{equation*}
\end{corollary}

\begin{proof} 
In light of Corollary~\ref{ORDER:Z}, used here with~$\delta:=\frac12$, we can find~$\theta\in(0,\min\{\rho,1\})$ and~$C\in(1,+\infty)$ such that, for all~$t\in(0,\theta]$,
\begin{equation*}
\ell(t)\le\frac{C}{\sqrt t}.\end{equation*}
Therefore, by Lemma~\ref{bfewERTYU3892gyf},
used here with~$c:=\frac12$,
\begin{eqnarray*}&&
\iint_{{U\times U^c}\atop{\{|x-y|\le\theta\}}}\mathcal{K}(x-y)\,dx\,dy=\iint_{{U\times U^c}\atop{\{|x-y|\le\theta\}}}\frac{\ell(|x-y|)}{|x-y|^{N}}\,dx\,dy
\\&&\qquad\le
\iint_{{U\times U^c}\atop{\{|x-y|\le\theta\}}}\frac{C}{|x-y|^{N+\frac12}}\,dx\,dy<+\infty
,\end{eqnarray*}
as desired.\end{proof}

\begin{corollary}\label{OJSL0oLK}
We have that
$$\iint_{{\Omega\times \Omega^c}}\mathcal{K}(x-y)\,dx\,dy<+\infty.$$
\end{corollary}

\begin{proof}
Let~$\theta$ be as in 
Corollary~\ref{coro:bfewERTYU3892gyf}. Then,
\begin{equation*}
\iint_{{\Omega\times \Omega^c}\atop{\{|x-y|\le\theta\}}}\mathcal{K}(x-y)\,dx\,dy<+\infty.
\end{equation*}
Also, thanks to assumption~$(K_1)$,
\begin{eqnarray*}
\iint_{{\Omega\times \Omega^c}\atop{\{|x-y|>\theta\}}}\mathcal{K}(x-y)\,dx\,dy
\le  |\Omega|\int_{\R^N\setminus B_\theta}\mathcal{K}(z)\,dz
<+\infty.\end{eqnarray*}
The desired result now plainly follows.
\end{proof}

\section{Proofs of Theorems \ref{th1} and \ref{th2}} \label{section3}

We introduce a notation that will be useful throughout the paper. Given $A$, $B \subset \R^{N}$ we write
\begin{equation} \label{notation}
u(A,B):=\iint_{A \times B} (u(x) - u(y))^2 \mathcal{K}(x- y) \, dx \, dy.
\end{equation}

We start this section with the simple observation that
the Neumann condition decreseas the Gagliardo seminorm of a function.

\begin{lemma} \label{lemma1}
Let $u : \R^N \to \R$ with $u|_{\Omega} \in L^1(\Omega)$. For every $x \in \R^N \setminus \Omega$, define
\begin{equation*}
E_u(x) := \int_{\Omega} u(z)\mathcal{K}(x-z) \, dz
\end{equation*} and
\begin{equation} \label{eq21}
\widetilde{u}(x) := 
\begin{cases} 
u(x), & \text{if } x \in \Omega \cup \D, \\ \displaystyle
\frac{E_u(x)}{E_1(x)}, & \text{if } x \in \nN,
\end{cases}
\end{equation}

Then,
\begin{equation*}
\iint_{\cQ} (\widetilde{u}(x) - \widetilde{u}(y))^2 \mathcal{K}(x-y) \, dx \, dy 
\leq 
\iint_{\cQ} (u(x) - u(y))^2 \mathcal{K}(x-y) \, dx \, dy.
\end{equation*}
Moreover, equality in the above formula holds if and only if $u$ satisfies the condition 
  \begin{equation*} 
 \cN u(x) := \int_{\Omega} \left(u(x) - u(y)\right) \mathcal{K}(x-y) \, dy = 0, \quad \text{for every } x \in \R^N \setminus \Omega.
 \end{equation*}
\end{lemma}

\begin{proof}
The proof is similar to the one in \cite[Theorem 2.1]{MR4651677}, but with a different kernel and renormalizing the sequence only in the region $\nN$. We provide the full details for the convenience of the reader.

We notice that it is not restrictive to assume that
\begin{equation*}
\iint_{\cQ} |u(x) - u(y)|^2 \mathcal{K}(x-y) \, dx \, dy <+ \infty
\end{equation*}
otherwise the desired result would be trivially true. 

We observe that
\begin{equation*}
\iint_{\Omega  \times \Omega } |\widetilde{u}(x) - \widetilde{u}(y)|^2 \mathcal{K}(x- y) \, dx \, dy 
= 
\iint_{\Omega \times \Omega} |u(x) - u(y)|^2 \mathcal{K}(x-y) \, dx \, dy
\end{equation*}
and
\begin{equation*} 
\iint_{\D \times \Omega} |\widetilde{u}(x) - \widetilde{u}(y)|^2 \mathcal{K}(x-y) \, dx \, dy 
= 
\iint_{\D \times \Omega} |u(x) - u(y)|^2 \mathcal{K}(x-y) \, dx \, dy.
\end{equation*}
Therefore,
in the notation of~\eqref{notation}, we have that
\begin{equation}\label{vbcjwquegfyw4ftwi76543}\begin{split}
    \iint_{\cQ} |\widetilde u(x) - \widetilde u(y)|^2 \mathcal{K}(x-y) \, dx \, dy & = \widetilde u(\Omega, \Omega) 
 + 2 \widetilde u (\Omega, \D) + 2\widetilde u (\Omega, \nN)\\&=u(\Omega, \Omega) 
 + 2 u (\Omega, \D) + 2\widetilde u (\Omega, \nN). 
\end{split}\end{equation}

We now define~$\varphi:= u - \widetilde{u}$ and we see that, for a.e.~$y \in \nN$,
\begin{eqnarray*}&&
\int_{\Omega} (u(x) - u(y))^2 \mathcal{K}(x-y) \, dx =
\int_{\Omega} (u(x) - \widetilde{u}(y) - \varphi(y))^2 \mathcal{K}(x-y) \, dx\\
&&\quad=\int_{\Omega} (u(x) - \widetilde{u}(y))^2 \mathcal{K}(x-y) \, dx - 2\varphi(y) \int_{\Omega} (u(x) - \widetilde{u}(y)) \mathcal{K}(x-y) \, dx\\&&\qquad\quad + 
\int_{\Omega} \varphi(y)^2 \mathcal{K}(x-y) \, dx
.
\end{eqnarray*}
By the definition of $\widetilde{u}$, we have that, for a.e.~$y \in \nN$,
\begin{equation*}
\int_{\Omega} (u(x) - \widetilde{u}(y)) \mathcal{K}(x-y) \, dx = E_u(y) - \frac{E_u(y)}{E_1(y)} E_1(y) = 0.
\end{equation*}
Accordingly, for a.e.~$y \in \nN$,
\begin{equation*}
\int_{\Omega} (u(x) - u(y))^2 \mathcal{K}(x-y) \, dx 
= 
\int_{\Omega} (\widetilde{u}(x) - \widetilde{u}(y))^2 \mathcal{K}(x-y) \, dx 
+ 
\int_{\Omega} \varphi(y)^2 \mathcal{K}(x-y) \, dx.
\end{equation*}
Integrating over $\nN$, we conclude that
\begin{equation*}
\int_{\nN }\int_{ \Omega} (u(x) - u(y))^2 \mathcal{K}(x-y) \, dx \, dy 
\geq 
\int_{\nN }\int_{ \Omega} (\widetilde{u}(x) - \widetilde{u}(y))^2 \mathcal{K}(x-y) \, dx \, dy,
\end{equation*}
with equality if and only if $\varphi \equiv 0$. 

Plugging this information into~\eqref{vbcjwquegfyw4ftwi76543}
we obtain the desired result.
\end{proof}

The hypotheses $(f_1)$–$(f_2)$ on the nonlinearity are not sufficient to guarantee that the functional $\mathcal{E}$ is well defined in $\HOOJ(\Omega \cup \nN)$, since the space $\HOOJ(\Omega \cup \nN)$ embeds continuously only in $L^2(\Omega \cup \nN)$ and the set $\nN$ may be unbounded.  To address this issue, we will ``truncate'' the functions that
we are working with from above at a certain level $L$, where $L$ is given by hypothesis~$(f_1)$. Since we will work with minimizing sequences, it is important that this procedure does not increase the Gagliardo seminorm, and the following lemma shows that this is indeed the case.

\begin{lemma} \label{lemma2}
    Let $u \in \HOOJ(\Omega \cup\nN)$, $L>0$ and~$\widetilde{u}(x):= \min \left\{u(x),L \right\}$.
    
     Then, 
    \begin{equation*}
\iint_\cQ (\widetilde{u}(x)-\widetilde{u}(y))^2 \mathcal{K}(x-y) \, dx \, dy \leq \iint_\cQ (u(x)-u(y))^2 \mathcal{K}(x-y) \, dx \, dy 
.\end{equation*}
 \end{lemma}

\begin{proof} By inspection of all possible cases, one sees that
\begin{equation*} |\widetilde{u}(x)-\widetilde{u}(y)|\leq |u(x)-u(y)|, \end{equation*} The desired result plainly follows.
\end{proof}

We already pointed out in Remark \ref{compactness} that the space $\HOOJ(\Omega\cup\nN)$ is not compactly embedded into any Lebesgue space, which creates convergence issues when one aims to prove the existence of minimizers of the functional~$\mathcal{E}$. On the other hand, if one restricts to functions satisfying Dirichlet boundary conditions on $\nN \cup \D$, one can work in $\HOOJ(\Omega)$, which is compactly embedded into $L^2(\Omega)$.

To implement this strategy, from now on, given $u \in \HOOJ(\Omega \cup \nN)$ we will use the notation
\begin{equation*}
    \overline{u}(x):= \begin{cases}
    \min \left\{u(x),L \right\} &{\mbox{ if }} x \in \Omega, \\
    u & {\mbox{ if }}x \in \nN,
\end{cases}
\end{equation*}
where $L$ is as in Hypothesis $(f_1)$.

\begin{lemma} \label{lemma3}
    Let $(u_k)_k \subset \HOOJ(\Omega \cup \nN)$ be such that
\begin{equation*} \sup_{k \in \N} \Vert u_k \Vert_{L^\infty (\Omega)}+
\iint_{\Omega\times\Omega} (u_k(x)-u_k(y))^2 \mathcal{K}(x-y) \, dx \, dy < +\infty.
\end{equation*}    
Let
    \begin{equation}\label{cnmssdf4it765defiukkk}
        v_k(x):=\begin{cases}
            u_k(x) &{\mbox{ if }} x \in \Omega \cup\D,\\
            0 &{\mbox{ if }} x \in \nN.
        \end{cases}
    \end{equation}
    
    Then there exists a constant $C >0$, depending on~$\mathcal{K}$ and~$\Omega$ such that
    \begin{equation*}\begin{split}&
        \iint_\cQ (v_k(x)-v_k(y))^2 \mathcal{K}(x-y) \, dx \, dy \\&\qquad\leq \iint_{\Omega\times\Omega} (u_k(x)-u_k(y))^2 \mathcal{K}(x-y) \, dx \, dy +
      C \sup_{k \in \N} \Vert u_k \Vert_{L^\infty (\Omega)}^2.
      \end{split}
    \end{equation*}
    
    \begin{proof}
 We observe that   
    \begin{equation}\label{vnmdwfui4yt734ty3489tguobhsuckj}
\begin{split}
      &  \iint_\cQ (v_k(x)-v_k(y))^2 \mathcal{K}(x-y) \, dx \, dy \\ &\qquad= \iint_{\Omega \times \Omega} (u_k(x)-u_k(y))^2 \mathcal{K}(x-y) \, dx \, dy 
        + 2 \int_{ \Omega} \int_{\D \cup \nN} u_k^2(x) \mathcal{K}(x-y) \, dx \, dy \\&\qquad
          \leq \iint_{\Omega \times \Omega} (u_k(x)-u_k(y))^2 \mathcal{K}(x-y) \, dx \, dy + \sup_{k \in \N} \Vert u_k \Vert_{L^\infty (\Omega)}^2
           \int_{ \Omega} \int_{\D \cup \nN}  \mathcal{K}(x-y) \, dx \, dy 
.\end{split}    \end{equation}

Moreover, we let~$\theta$ be as in Corollary~\ref{coro:bfewERTYU3892gyf}.
Then,\begin{eqnarray*}
\iint_{ {\Omega \times\Omega^c }\atop{\{|x-y|\le\theta\}}}  \cK(x-y) \, dx \, dy <+\infty.
\end{eqnarray*}
Furthermore, 
    \begin{eqnarray*}
&&\iint_{ {\Omega \times\Omega^c }\atop{\{ |x-y| > \theta\}}}  \mathcal{K}(x-y) \, dx \, dy\leq 
\iint_{ \Omega \times (\R^N\setminus B_\theta)} \mathcal{K}(z) \, dx \, dz
\\&&\qquad\leq|\Omega|\int_{ \R^N\setminus B_\theta} \mathcal{K}(z)  \, dz
, 
    \end{eqnarray*}
which is finite, thanks to $(K_1)$.

Gathering these pieces of information, we conclude that
    \begin{eqnarray*}&&
   \int_{ \Omega} \int_{\D \cup \nN}  \mathcal{K}(x-y) \, dx \, dy =  \int_{ \Omega} \int_{\Omega^c}  \mathcal{K}(x-y) \, dx \, dy<+\infty.
\end{eqnarray*}    
From this and~\eqref{vnmdwfui4yt734ty3489tguobhsuckj} we obtain the desired estimate.
    \end{proof}
\end{lemma}

\begin{proposition} \label{prop1}
Let $(v_k)_k \subset \HOOJ(\Omega)$ be a non-negative sequence such that
\begin{equation}\label{cxzmzfeuit7489tdiwoscnka}
 \sup_{k\in\N}\iint_\cQ (v_k(x)-v_k(y))^2 \mathcal{K}(x-y) \, dx \, dy <+\infty
\end{equation}
and let
\begin{equation*}
    \overline{v}_k(x):= \begin{cases}
    \min \left\{v_k(x),L \right\} &{\mbox{ if }} x \in \Omega, \\
    v_k =0& {\mbox{ if }}x \in \nN.\end{cases}
\end{equation*}

Then, there exists $\overline{v} \in L^2(\Omega)$ such that $\overline{v}_k \to \overline{v}$ in $L^q(\Omega)$ as $k \to+ \infty$, for all $q \in \left[2,+\infty \right)$.
\end{proposition}

\begin{proof}
We first show that the sequence $\overline{v}_k$ is uniformly bounded in $\HOOJ(\Omega)$,
namely that
\begin{equation}\label{hcjdsif34t7itjkas7685i34u8hybtrg}
\sup_{k\in\N}\int_\Omega \overline{v}_k^2\,dx+
\iint_{{\cQ}} \left(\overline{v}_k(x) - \overline{v}_k(y)\right)^2 \mathcal{K}(x- y) \, dx \, dy <+ \infty.
\end{equation}
To check this, we observe that
$$ \sup_{k\in\N}\int_\Omega \overline{v}_k^2\,dx\le \int_\Omega L\,dx=L|\Omega|.
$$
Moreover, from Lemma \ref{lemma2} and~\eqref{cxzmzfeuit7489tdiwoscnka}, it follows that
\[\sup_{k\in\N}
\iint_\cQ (\overline{v}_k(x)-\overline{v}_k(y))^2 \mathcal{K}(x-y) \, dx \, dy \leq 
\iint_\cQ ({v}_k(x)-{v}_k(y))^2 \mathcal{K}(x-y) \, dx \, dy
 <+\infty.
\]
These observations establish~\eqref{hcjdsif34t7itjkas7685i34u8hybtrg}.

As a consequence of~\eqref{hcjdsif34t7itjkas7685i34u8hybtrg}, there exists~$v \in \HOOJ(\Omega)$ such that $\overline{v}_k \rightharpoonup v$ in $\HOOJ(\Omega)$ and $\overline{v}_k \to v$ in $L^2(\Omega)$ as $k \to +\infty$ (see \cite[Theorem 2.1]{MR3759570}).

We notice that~$\overline{v}\le L$ in~$\Omega$.
Therefore, for all~$q \in \left[2,+\infty \right)$,
\begin{equation*}
   \lim_{k\to+\infty}  \int_\Omega | \overline{v}_k - \overline{v} |^q \, dx \leq \left(2L\right)^{q-2}\lim_{k\to+\infty} \int_\Omega | \overline{v}_k - \overline{v} |^2 \, dx = 0
,\end{equation*}
and this completes the proof.\end{proof}

Now we are ready to prove Theorem \ref{th1}.

 \begin{proof}[Proof of Theorem \ref{th1}]
Let~$(u_k)_k \subset \HOOJ(\Omega \cup \nN)$ be a minimizing sequence for the functional $\cE$. From the triangular inequality
\begin{equation*}
   \iint_{\cQ} \left(|u(x)|-|u(y)|\right)^2 \mathcal{K}(x-y) \,dx dy \leq  \iint_{\cQ} \left(u(x)-u(y)\right)^2 \mathcal{K}(x-y) \,dx dy
\end{equation*}
and assumption~$(f_2)$ it follows that~$\mathcal{E}(|u|)\le\mathcal{E}(u)$, and therefore we can suppose that all the elements of the sequence~$(u_k)_k$ are non-negative. 

Moreover, defining~$\widetilde{u}_k(x):=\min\{u_k(x),L\}$,
we deduce from Lemma \ref{lemma2} that
$$   \iint_{\cQ} \left(\widetilde{u}_k(x)-\widetilde{u}_k(y)\right)^2 \mathcal{K}(x-y) \,dx dy \leq  \iint_{\cQ} \left(u_k(x)-u_k(y)\right)^2 \mathcal{K}(x-y) \,dx dy
$$
and from assumption~$(f_1)$ that
\begin{eqnarray*}&&
 \int_\Omega F(x,\widetilde{u}_k) \, dx 
 = \int_{\Omega \cap \{u_k(x)< L\}}F(x,\widetilde{u}_k) \, dx 
+\int_{\Omega \cap \{u_k(x)\ge L\}} F(x,\widetilde{u}_k) \, dx 
\\&&\qquad=\int_{\Omega \cap \{u_k(x)\le L\}}F(x,u_k) \, dx 
+\int_{\Omega \cap \{u_k(x)> L\}} F(x,L) \, dx
\geq \int_\Omega F(x,u_k) \, dx.
\end{eqnarray*}
In light of these considerations, we can also suppose that
all the elements of the sequence~$(u_k)_k$ are bounded above by~$L$. 

As a result, in what follows we will assume that the sequence~$(u_k)_k$ is made of non-negative functions such that
\begin{equation}\label{cdms43iyt89ghlewwet58765}
\sup_{k\in\N} \|u_k\|_{L^\infty(\Omega)}\le L.\end{equation}

We also notice that $0 \in \HOOJ(\Omega \cup \nN)$, so without loss of generality we can assume that
$$\sup_{k\in\N}\cE(u_k) \leq 0.$$ This, together with~\eqref{cdms43iyt89ghlewwet58765} and assumption~$(f_1)$, entails that
\begin{equation} \label{eq6}\begin{split}&\sup_{k\in\N}
\iint_{\cQ} \left(u_k(x)-u_k(y)\right)^2\mathcal{K}(x-y) \,dx dy
=\sup_{k\in\N}\mathcal{E}(u_k)+ \int_\Omega F(x,u_k) \, dx
\\&\qquad\leq\sup_{k\in\N} \int_\Omega F(x,u_k) \, dx \leq \int_\Omega A(x) \, dx <+ \infty.\end{split}
\end{equation}

We now consider the sequence~$v_k$ defined as in~\eqref{cnmssdf4it765defiukkk} and we use Lemma~\ref{lemma3} to find that
    \begin{equation*}\begin{split}&
        \iint_\cQ (v_k(x)-v_k(y))^2 \mathcal{K}(x-y) \, dx \, dy \\&\qquad\leq \iint_{\Omega\times\Omega} (u_k(x)-u_k(y))^2 \mathcal{K}(x-y) \, dx \, dy +
      C \sup_{k \in \N} \Vert u_k \Vert_{L^\infty (\Omega)}^2,
      \end{split}
    \end{equation*}
for some~$C >0$, depending on~$\mathcal{K}$ and~$\Omega$.
    
From this, \eqref{cdms43iyt89ghlewwet58765} and~\eqref{eq6}, we obtain that
\begin{equation*}
\sup_{k\in\N}
\iint_{\cQ} \left(v_k(x)-v_k(y)\right)^2\mathcal{K}(x-y) \,dx dy<+\infty .
\end{equation*}
The uniform bound in \eqref{cdms43iyt89ghlewwet58765} also gives that the sequence $(v_k)_k$ is uniformly bounded in~$L^2(\Omega)$. Hence, $(v_k)_k$ is uniformly
bounded in~$\HOOJ(\Omega)$, and up to a subsequence, there exists~$u \in \HOOJ(\Omega)$ such that~$v_k \rightharpoonup u$ in~$\HOOJ(\Omega)$ as~$k \to +\infty$.

Furthermore, exploiting Proposition \ref{prop1}, we deduce that~$v_k \to u$ in $L^q(\Omega)$ as~$k\to+\infty$
for every $q \in [2, +\infty)$. Since $u_k = v_k$ in $\Omega$, we also have that $u_k \to u$ in $L^q(\Omega)$ as $k \to +\infty$
for every $q \in [2, +\infty)$.

As a consequence, we have that $u_k \to u$ a.e. in $\Omega$. Hence, since $F$ is a Carathéodory function, we get that $F(x,u_k) \to F(x,u)$ a.e. in $\Omega$. Thus, using $(f_1)$ and applying the Dominated Convergence Theorem, we obtain that
\begin{equation}\label{cdsr43DGHJ76322878768uyt}
\lim_{k \to +\infty} \int_\Omega F(x,u_k) \, dx = \int_{\Omega} F(x,u) \, dx.
\end{equation}

Furthermore, since~$\Omega$ is bounded, we gather that~$u_k \to u$ in $L^1(\Omega)$ as~$k\to+\infty$. Consequently, for all~$x \in \nN$,
\begin{equation}\label{cnmsui43t7834542345678vcd}
\lim_{k\to+\infty}\frac{\displaystyle \int_{ \Omega} u_k(z)\mathcal{K}(x-z) \, dz}{ \displaystyle \int_{\Omega} \mathcal{K}(x-z) \, dz} = \frac{\displaystyle \int_{ \Omega} u(z)\mathcal{K}(x-z) \, dz}{ \displaystyle \int_{\Omega} \mathcal{K}(x-z) \, dz}.
\end{equation}
Accordingly, we can re-define~$u_k$ and~$u$ in the set~$ \nN$
as
\begin{eqnarray*}u_k(x)&:= & \frac{\displaystyle \int_{ \Omega} u_k(z)\mathcal{K}(x-z) \, dz}{ \displaystyle \int_{\Omega} \mathcal{K}(x-z) \, dz} \\
{\mbox{and }}\qquad
u(x)&:=&  \frac{\displaystyle \int_{ \Omega} u(z)\mathcal{K}(x-z) \, dz}{ \displaystyle \int_{\Omega} \mathcal{K}(x-z) \, dz}
\end{eqnarray*}
and deduce from~\eqref{cnmsui43t7834542345678vcd} that~$u_k$ converges to~$u$ a.e. in~$\nN$ as~$k\to+\infty$.

We point out that this procedure maintains the minimality
of the sequence~$u_k$, thanks to Lemma~\ref{lemma1}.

Gathering these pieces of information, and 
using~\eqref{cdsr43DGHJ76322878768uyt} and Fatou's Lemma, we obtain that
\begin{equation*}
\cE(u) \leq \liminf_{k \to \infty} \left( \frac{1}{4} \iint_{\cQ} \left(u_k(x)-u_k(y)\right)^2\mathcal{K}(x-y) \,dx dy-\int_\Omega F(x,u_k) \, dx  \right).
\end{equation*}
As a result, $u$ is a minimizer for $\cE$ in $\HOOJ(\Omega \cup \nN)$.
\end{proof}

We now focus our attention the specific case of problem \eqref{mainlogistic}. 
We recall that the functional associated with problem \eqref{mainlogistic} takes the following explicit form
\begin{equation*}
    \cE(u) := \frac{1}{4}\iint_{\cQ} \left(u(x)-u(y)\right)^2 \mathcal{K} (x-y) \,dx dy - \int_\Omega \frac{m(x)}{2} u^2 - \frac{\mu(x)}{3} |u|^3 \, dx.
\end{equation*}

\begin{proposition} \label{prop2}
Let $m \in L^\infty(\Omega)$ and $\mu \in L^1(\Omega)$, and suppose there exists a constant $\overline{\mu} > 0$ such that $\mu(x) \geq \overline{\mu}$ a.e. in $\Omega$. Then the function 
\[
F(x,t) := \frac{m(x)}{2} t^2 - \frac{\mu(x)}{3} |t|^3
\]
satisfies conditions $(f_1)$–$(f_2)$.
\end{proposition}

\begin{proof}
Let~$L > 0$ sufficiently large such that
\begin{equation} \label{eq15}
    \frac{\|m^+\|_{L^\infty(\Omega)}}{2} \leq \frac{\overline{\mu}}{3}L.
\end{equation}
We note that, for all $t \in [0, L]$,
\[
    |F(x,t)| \leq \frac{|m(x)|}{2}L^2 + \frac{\mu(x)}{3}L^3 \in L^1(\Omega).
\]
Moreover, for all $t > L$ we have
\begin{align*}
    F(x,t) &= \frac{m(x)}{2}t^2 - \frac{\mu(x)}{3}t^3 \\
           &\leq \frac{m(x)}{2}t^2 - \frac{\mu(x)}{3}Lt^2 \\
           &= -\frac{m^-(x)}{2}t^2 + \left( \frac{m^+(x)}{2} - \frac{\mu(x)}{3}L \right)t^2.
\end{align*}
In view of \eqref{eq15}, we have that 
$$\frac{m^+(x)}{2} - \frac{\mu(x)}{3}L \leq
\frac{\|m^+\|_{L^\infty(\Omega)}}{2} -\frac{\overline{\mu}}{3}L
\leq 0,$$ so
\[
    F(x,t) \leq -\frac{m^-(x)}{2}L^2 + \left( \frac{m^+(x)}{2} - \frac{\mu(x)}{3}L \right)L^2 = F(x,L),
\]
which completes the proof of assumption~$(f_1)$.

Moreover,
condition $(f_2)$ is trivially satisfied since $F(x,|t|) = F(x,t)$.
\end{proof}

 \begin{proof}[Proof of Theorem \ref{th2}]
By Theorem \ref{th1} and Proposition \ref{prop2}, there exists a non-negative solution to \eqref{mainlogistic}.

We first prove statement (i). Assume $m \leq 0$, and suppose by contradiction that there exists a non-negative and non-trivial solution~$u $ of \eqref{mainlogistic}. Since $\mu \geq \overline{\mu} > 0$ in $\Omega$ and $u \not\equiv 0$, we have
\[
\int_\Omega \mu(x) u^3\,dx > 0.
\]
Therefore, since~$u$ is a solution of~\eqref{mainlogistic},
\[
0 \leq \iint_{\cQ} (u(x) - u(y))^2 \, \mathcal{K}(x-y)\,dx\,dy = \int_\Omega m(x) u^2\,dx - \int_\Omega \mu(x) u^3\,dx <0,\]
which gives a contradiction. Thus, any solution must be trivial in this case, and the claim in~(i) is proved.

We now address (ii). We claim that the minimizer found in Theorem \ref{th1} cannot be trivial. To this end, 
we will construct a competitor with strictly negative energy.

To implement this strategy, we pick a function~$\phi\in C^\infty_0(\overline{\Omega}\cup\nN, [0,1])$ such that~$\phi=1$ in~$\Omega$. For~$\e\in(0,1)$, let also~$\phi_\e:=\e\phi$.
We point out that~$\phi_\varepsilon=\varepsilon$ in~$\Omega$ and~$\phi_\varepsilon=0$ in~$\D$, thanks to the assumptions
in~\eqref{domain}. Therefore,
$\phi_\varepsilon$ is a competitor for the minimization of~$\mathcal{E}$.

We also observe that~$ \|\phi_\e\|_{L^\infty(\R^N)}\le\e$
and, for all~$x$, $y\in\R^N$, $|\phi_\e(x)-\phi_\e(y)|\le\e$.
Therefore,
\begin{eqnarray*}
&& \iint_{\cQ} (\phi_\varepsilon(x) - \phi_\varepsilon(y))^2 \, \mathcal{K}(x-y)\,dx\,dy \\
&& \qquad =  2\iint_{\Omega\times \Omega^c} (\phi_\varepsilon(x) - \phi_\varepsilon(y))^2 \, \mathcal{K}(x-y)\,dx\,dy\\
&&\qquad \le 2
\varepsilon^2\iint_{\Omega\times\Omega^c}  \mathcal{K}(x-y)\,dx\,dy. 
\end{eqnarray*}

To simplify the notation, we define
\begin{eqnarray*}
c_1&:=&\iint_{\Omega\times\Omega^c}  \mathcal{K}(x-y)\,dx\,dy,\\
c_2 &:=& \int_\Omega m(x)\,dx \qquad \text{and} \qquad c_3 := \int_\Omega \mu(x)\,dx.
\end{eqnarray*}
With this notation, we find that
\begin{eqnarray*}
\cE(\phi_\varepsilon) &=&\frac{1}{4}\iint_{\cQ} \left(\phi_\e(x)-\phi_\e(y)\right)^2 \mathcal{K} (x-y) \,dx dy - \int_\Omega \frac{m(x)}{2} \phi_\e^2 - \frac{\mu(x)}{3} |\phi_\e|^3 \, dx\\
&\leq& \frac{\e^2 c_1}2
-\frac{\e^2 c_2}2 +\frac{\e^3 c_3}3 
\\&=&\frac12\left( c_1
-{c_2}\right)\e^2+\frac{\e^3}3 c_3.
\end{eqnarray*}

In light of~\eqref{eq38}, we have that~$  c_1-{c_2}<0$,
hence, for $\varepsilon$ sufficiently
small, it follows that~$\cE(\phi_\varepsilon) < 0$, as desired.\end{proof}

\section{Spectral analysis of two weighted eigenvalue problems} \label{SpecAna} \label{section4}

In this section, we focus on the study of two weighted eigenvalue problems with mixed Dirichlet–Neumann boundary conditions. More precisely, we are interested in determining the smallest~$\lambda \in \mathbb{R}$ for which the problems
\begin{equation}
\label{eigenproblem}
\tag{$P_\lambda$}
\begin{cases}  
\mathfrak{L} u = \lambda m(x) u & \text{in } \Omega, \\
\mathcal{N} u(x) = 0 & \text{on } \mathfrak{N}, \\
u = 0 & \text{on } \mathfrak{D},
\end{cases}
\end{equation}
and 
\begin{equation}
\label{eigenschro}
\tag{$\widetilde{P}_\lambda$}
\begin{cases}  
\mathfrak{L} u - m(x)u = \lambda u & \text{in } \Omega, \\
\mathcal{N} u(x) = 0 & \text{on } \mathfrak{N}, \\
u = 0 & \text{on } \mathfrak{D},
\end{cases}
\end{equation}
admit non-trivial solutions.

Here above and in what follows, we assume that~$m : \Omega \to \mathbb{R}$ is a measurable function satisfying suitable assumptions and the set~$\mathfrak{N}$ is bounded. 

Moreover, we focus on the specific case
\begin{equation*}
\mathcal{K}(z) := \frac{ \chi_{B_1}(z)}{|z|^{N}}.
\end{equation*}

To address problem~\eqref{eigenproblem}, we consider the minimization of the functional~$Q:\HOO(\Omega \cup \mathfrak{N})\to\R$ defined as
\begin{equation*}
Q(u) := \frac12\iint_{ \mathcal{Q} \cap\{ |x-y| \leq 1\}} \frac{(u(x)-u(y))^2}{|x-y|^N} \,dx \,dy
\end{equation*}
subject to the constraint
\[
\int_\Omega m(x) u^2 \, dx = 1.
\]

Similarly, to deal with problem~\eqref{eigenschro}, we consider the minimization of the functional~$W:\HOO(\Omega \cup \mathfrak{N})\to\R$ defined as
\begin{equation*}
W(u) := \frac12\iint_{ \mathcal{Q} \cap\{ |x-y| \leq 1\}} \frac{(u(x)-u(y))^2}{|x-y|^N} \,dx \,dy - \int_\Omega m(x) u^2 \, dx
\end{equation*}
subject to the constraint
\[
\int_\Omega u^2 \, dx = 1.
\]

We observe that any constrained critical point $u \in \HOO(\Omega \cup \mathfrak{N})$ of the functional $Q$ is a weak solution of \eqref{eigenproblem}, namely,
for all~$w \in \HOO(\Omega \cup \mathfrak{N})$,
\begin{equation*}
\frac12\iint_{\mathcal{Q} \cap\{ |x-y| \leq 1\}} \frac{(u(x) - u(y))(w(x) - w(y))}{|x-y|^N} \, dx \, dy 
= \lambda \int_\Omega m(x) u(x) w(x) \, dx.
\end{equation*}
  
Analogously, any constrained critical point $u \in \HOO(\Omega \cup \mathfrak{N})$ of the functional $W$ is a weak solution of \eqref{eigenschro}, that is, for all $w \in \HOO(\Omega \cup \mathfrak{N})$,
\begin{equation*}
\frac12\iint_{\mathcal{Q} \cap\{ |x-y| \leq 1\}} \frac{(u(x) - u(y))(w(x) - w(y))}{|x-y|^N} \, dx \, dy 
- \int_\Omega m(x) u(x) w(x) \, dx
= \lambda \int_\Omega u(x) w(x) \, dx
.\end{equation*}
  
Next result states that if we take a sequence bounded in $\HO(\Omega)$, then it is possible to redefine the sequence in the region $\nN$ through a normalization with the kernel, and this new sequence is bounded in $\HOO(\Omega \cup \nN)$ provided that the regions $\nN$ and $\Omega$ do not touch.

\begin{proposition} \label{boundedness}
Suppose that~$\nN$ is bounded and that there exists $d_0>0$ such that $\dist(\Omega,\mathfrak{N}) \geq d_0 > 0$. 

Let $(u_k)_k \subset \HOO(\Omega \cup \nN)$ be such that
    \begin{equation*}
      \sup_{k\in\N}  \int_\Omega u_k^2 \, dx 
      +\iint_{ \cQ \cap\{ |x - y| \leq 1\}} \frac{(u_k(x)-u_k(y))^2}{|x-y|^N}\,dx\,dy<+\infty\end{equation*}
      and let
      \begin{equation}\label{23456defvk00} v_k(x):=\begin{cases}
u_k(x) & \text{if } x \in \Omega \cup \D, \\ \displaystyle
\frac{\displaystyle\int_{\Omega\cap\{|x-z|\le1\}} \frac{u_k(z)}{|x-z|^N} \, dz}{\displaystyle\int_{\Omega\cap\{|x-z|\le1\}} \frac{ dz}{|x-z|^N}} & \text{if } x \in \nN.
\end{cases}
\end{equation}
 
Then, the sequence $(v_k)_k $ is uniformly
bounded in $\HOO(\Omega \cup \nN)$, namely
$$
\sup_{k\in\N}\int_\Omega {v}_k^2\,dx+
 \iint_{ { \R^{2N} \setminus (\D \times \D)}\atop{\{|x-y| \leq 1\}}} \frac{(v_k(x) - v_k(y))^2}{|x-y|^N} 
\, dx \, dy <+ \infty.
$$
\end{proposition}

\begin{proof} 
To simplify the notation, we suppose that
$$ \sup_{k\in\N}  \left(\int_\Omega u_k^2 \, dx\right)^{1/2} \le K_0.$$
Then, by the definition of~$v_k$, we have that
\begin{equation}\label{87654hdjsfgvwejhSDFGH876500} \sup_{k\in\N}\int_\Omega {v}_k^2\,dx
=\sup_{k\in\N}\int_\Omega {u}_k^2\,dx\le K_0^2.\end{equation}
In addition, from H\"older's inequality, it follows that
\begin{equation} \label{eq16}
    \left|\int_\Omega v_k \, dx \right| \leq |\Omega|^{\frac{1}{2}} \left( \int_\Omega v_k^2 \right)^{\frac{1}{2}} \leq |\Omega|^{\frac{1}{2}}K_0.
\end{equation}

Moreover, for all $x \in \nN$,
\begin{equation}\label{cnmru3t25647SDFH765}
    \int_{\Omega \cap\{ |x-z| \leq 1\}} \frac{dz}{|x-z|^N} \geq |\Omega|.
\end{equation}
Combining this and~\eqref{eq16}, and recalling that $\dist(\Omega,\nN) \geq d_0$, we deduce that, for all~$x \in \nN$,
\begin{equation*} 
    v_k(x) = \frac{ \displaystyle\int_{\Omega \cap\{ |x-z| \leq 1\}} \frac{u_k(z)}{|x-z|^N} \, dz}{ \displaystyle \int_{ \Omega \cap\{ |x-z| \leq 1\}} \frac{dz}{|x-z|^N} } \leq \frac{\displaystyle \int_\Omega u_k \, dz}{\displaystyle |\Omega| d_0^N}
    =\frac{\displaystyle \int_\Omega v_k \, dz}{\displaystyle |\Omega| d_0^N} \leq \frac{K_0}{\displaystyle |\Omega|^{\frac{1}{2}} d_0^N} .
\end{equation*}
As a consequence, $ v_k \in L^\infty(\nN)$ and 
\begin{equation} \label{eq18}
\sup_{k\in\N}\Vert v_k \Vert_{L^\infty (\nN)} \leq \frac{K_0}{ |\Omega|^{\frac{1}{2}} d_0^N }\end{equation}

We remark that this also entails that
the sequence~$(v_k)_k$ is uniformly
bounded in $L^2( \nN)$, and thus in $L^2(\Omega \cup \nN)$.

We now claim that 
\begin{equation}\label{87654hdjsfgvwejhSDFGH87650}
  \sup_{k\in\N}  \iint_{ { \R^{2N} \setminus (\D \times \D)}\atop{\{|x-y| \leq 1\}}} \frac{(v_k(x) - v_k(y))^2}{|x-y|^N} \, dx \, dy<+\infty.
\end{equation}
To prove it, recalling the notation introduced in \eqref{notation}, we observe that  \begin{equation}\label{87654hdjsfgvwejhSDFGH8765}\begin{split}
&  \iint_{ { \R^{2N} \setminus (\D \times \D)}\atop{\{|x-y| \leq 1\}}} 
\frac{(v_k(x) - v_k(y))^2}{|x-y|^N} \, dx \, dy 
= v_k (\Omega \cup \nN,\Omega \cup \nN) +  2\,  v_k (\Omega \cup \nN, \D) \\
& \quad \quad = v_k(\Omega, \Omega) + 2 v_k(\Omega, \nN) + v_k(\nN, \nN) + 2 v_k(\Omega, \D) + 2v_k(\nN, \D) \\
& \quad \quad = v_k(\Omega, \Omega) + 2 v_k(\Omega, \D \cup \nN) + v_k(\nN, \nN) + 2 v_k(\nN, \D) \\
& \quad \quad = \iint_{ { \R^{2N} \setminus (
(\D \cup \nN) \times (\D \cup \nN))}\atop{\{ |x-y| \leq 1\}}} 
\frac{(v_k(x) - v_k(y))^2}{|x-y|^N} \, dx \, dy + v_k(\nN, \nN) + 2 v_k(\nN, \D) \\
& \quad  \quad = \iint_{\cQ\cap\{ |x-y| \leq 1\}} \frac{(v_k(x) - v_k(y))^2}{ |x-y|^N} \, dx \, dy + v_k(\nN, \nN) + 2 v_k (\nN, \D).
\end{split}\end{equation}

Now, from Lemma \ref{lemma1} we have that 
\begin{equation}\label{87654hdjsfgvwejhSDFGH87652}
   \iint_{\cQ\cap\{ |x-y| \leq 1\}} \frac{(v_k(x) - v_k(y))^2}{ |x-y|^N} \, dx \, dy \leq \iint_{ \cQ
   \cap\{ |x-y| \leq 1\}} \frac{(u_k(x) - u_k(y))^2}{|x-y|^N} \, dx \, dy<+\infty.
\end{equation}

Moreover, for any~$z \in \Omega$, we consider
the function $h_z(x) := |x - z|^{-N}$ and we see that, if~$x$, $y\in\nN$ with~$|x-y|\le d_0/2$, then, for all~$t\in(0,1)$,
$$ |y+t(x-y)-z|\ge |y-z|-|x-y|\ge d_0-\frac{d_0}2=\frac{d_0}2,$$
and therefore 
\begin{equation}\label{cds47387fiwhfkwSDFGH}\begin{split}
|h_z(x) - h_z(y)|&\le \int_0^1 |\nabla h_z(y+t(x-y)|\,|x-y|\,dt\\&= N \int_0^1 \frac{|x-y|}{|y+t(x-y)-z|^{N+1}}\,dt\\
&\le\frac{2^{N+1}N}{d_0^{N+1}}|x-y|.
\end{split}\end{equation}

We recall that, for all~$x$, $y\in\nN$,
\begin{eqnarray*}
    |v_k(x) - v_k(y)| & =
&    \left|\frac{\displaystyle\int_{\Omega\cap\{|x-z|\le1\}} u_k(z)h_z(x) \, dz}{\displaystyle\int_{\Omega\cap\{|x-z|\le1\}} h_z(x)\,dz } 
    -\frac{\displaystyle\int_{\Omega\cap\{|y-z|\le1\}} u_k(z)h_z(y) \, dz}{\displaystyle\int_{\Omega\cap\{|y-z|\le1\}} h_z(y)\, dz}\right|  .
\end{eqnarray*}
Thus, exploiting~\eqref{cnmru3t25647SDFH765} and~\eqref{cds47387fiwhfkwSDFGH}, we find that,
for all~$x$, $y\in\nN$ with~$|x-y|\le d_0/2$,
\begin{eqnarray*}
&& |\Omega|^2   |v_k(x) - v_k(y)| \\& \le
& 
\Bigg|
\left(\int_{\Omega\cap\{|x-z|\le1\}} u_k(z)h_z(x) \, dz\right)\left(\int_{\Omega\cap\{|y-z|\le1\}} h_z(y)\, dz \right)
 \\&&\qquad   -\left(\int_{\Omega\cap\{|y-z|\le1\}} u_k(z)h_z(y) \, dz\right)\left(\int_{\Omega\cap\{|x-z|\le1\}} h_z(x)\,dz\right)\Bigg| \\&\le&
\left(\int_{\Omega\cap\{|y-z|\le1\}} h_z(y)\, dz \right)
\left(\int_{\Omega}| u_k(z)|\,|h_z(x)-h_z(y)| \, dz\right)
 \\&&\qquad   +
 \left(\int_{\Omega\cap\{|y-z|\le1\}} |u_k(z)|\,h_z(y) \, dz\right)\left(\int_{\Omega} |h_z(x)-h_z(y)|\,dz\right) 
\\&\le&\frac{2^{N+1}N}{d_0^{N+1}}|x-y|
\left(\int_{\Omega\cap\{|y-z|\le1\}} h_z(y)\, dz \right)
\left(\int_{\Omega}| u_k(z)|\, dz\right)
 \\&&\qquad   +\frac{2^{N+1}N|\Omega|}{d_0^{N+1}}|x-y|
 \left(\int_{\Omega\cap\{|y-z|\le1\}} |u_k(z)|\,h_z(y) \, dz\right)\\
 &\le&\frac{2^{N+2}N|\Omega|}{d_0^{2N+1}}|x-y|
\left(\int_{\Omega}| u_k(z)|\, dz\right)
.
\end{eqnarray*}

Accordingly, since the sequence~$(u_k)_k$
is uniformly bounded in~$L^2(\Omega)$, and therefore in~$L^1(\Omega)$, we conclude that,
for all~$x$, $y\in\nN$ with~$|x-y|\le d_0/2$,
$$   |v_k(x) - v_k(y)| \le C|x-y|,$$
for some~$C>0$ independent of~$k$.

As a result,
\begin{eqnarray*}&&
\int_{{\nN\times\nN}\atop{\{|x-y|\le d_0/2\}}}\frac{ |v_k(x) - v_k(y)| ^2}{|x-y|^N}\,dx\,dy\le C
\int_{{\nN\times\nN}\atop{\{|x-y|\le d_0/2\}}}|x-y|^{2-N}\,dx\,dy
\end{eqnarray*}
whcih is finite, since~$\nN$ is a bounded set.

Moreover, exploiting~\eqref{eq18}, we have that
\begin{eqnarray*}&&
\int_{{\nN\times\nN}\atop{\{d_0/2<|x-y|\le 1\}}}\frac{ |v_k(x) - v_k(y)| ^2}{|x-y|^N}\,dx\,dy\le C
\int_{{\nN\times\nN}\atop{\{|x-y|\le d_0/2\}}}\frac{dx\,dy}{|x-y|^{N}}
\end{eqnarray*}
which is finite.

The last two displays give that
\begin{equation}\label{87654hdjsfgvwejhSDFGH87653}  v_k(\nN, \nN)<+\infty.\end{equation}

Furthermore, we deduce from~\eqref{eq18} 
and Lemma~\ref{bfewERTYU3892gyf} (used here with~$U:=\nN$ and~$c:=0$)
that
\begin{equation}\label{87654hdjsfgvwejhSDFGH87654}
v_k (\nN, \D)= \iint_{{\nN \times \D}\atop{\{ |x-y| \leq 1\}}} \frac{(v_k(x))^2}{ |x-y|^N} \, dx \, dy \leq \left(\frac{K_0}{ |\Omega|^{\frac{1}{2}} d_0^N} \right)^{2} \iint_{
{\nN \times \D}\atop{\{ |x-y| \leq 1\}}} \frac{dx\,dy}{ |x-y|^N} 
<+\infty.  
\end{equation}

Combining~\eqref{87654hdjsfgvwejhSDFGH8765}, \eqref{87654hdjsfgvwejhSDFGH87652}, \eqref{87654hdjsfgvwejhSDFGH87653}
and~\eqref{87654hdjsfgvwejhSDFGH87654},
we obtain the claim in~\eqref{87654hdjsfgvwejhSDFGH87650},
which, together with the estimate~\eqref{87654hdjsfgvwejhSDFGH876500},
completes the proof.
\end{proof}

We now establish a general Poincar\'e-type inequality that will be used throughout the remaining part of this section.

\begin{proposition} \label{poincare}Suppose that $\mathfrak{N}$ is bounded and that $\dist(\Omega,\mathfrak{N}) \geq d_0 > 0$. Then there exists a constant $C > 0$ such that
\[
\int_\Omega u^2 \, dx \;\leq\; C \iint_{\mathcal{Q} \cap\{ |x - y| \leq 1\}} 
\frac{(u(x) - u(y))^2}{|x - y|^N} \, dx \, dy,
\]
for every $u \in \HOO(\Omega \cup \mathfrak{N})$.
\end{proposition}

\begin{proof}
Suppose by contradiction that the claim is false. Then there exists a sequence $(u_k)_k\subset\HOO(\Omega\cup\nN)$ such that
\[
\int_\Omega u_k^2\,dx = 1
\qquad\text{and}\qquad
\iint_{ \cQ \cap\{ |x - y| \leq 1\}} \frac{(u_k(x)-u_k(y))^2}{|x-y|^N}\,dx\,dy \le \frac{1}{k}.
\]

By Proposition \ref{boundedness}, 
up to replacing~$u_k$ with~$v_k$ given in~\eqref{23456defvk00}, we may assume that $(u_k)_k$ is
uniformly bounded in $\HOO(\Omega\cup\nN)$, namely
$$
\sup_{k\in\N}\int_\Omega {u}_k^2\,dx+
 \iint_{ { \R^{2N} \setminus (\D \times \D)}\atop{\{|x-y| \leq 1\}}} \frac{(u_k(x) - u_k(y))^2}{|x-y|^N} 
\, dx \, dy <+ \infty.
$$
Hence, there exists~$u\in\HOO(\Omega\cup\nN)$ such that,
up to a subsequence, 
\[
u_k \rightharpoonup u \quad\text{weakly in }\HOO(\Omega\cup\nN)
\qquad{\mbox{and}}\qquad
u_k \to u \quad\text{strongly in }L^2(\Omega)
\]
(see \cite[Theorem 2.1]{MR3759570}).

These convergences imply that
\begin{equation}\label{cnmrui342345ASDFGH000}
\int_\Omega u^2\,dx = 1
\end{equation} and
\begin{equation}\label{cnmrui342345ASDFGH0002}
\iint_{ \cQ \cap\{ |x - y| \leq 1\}} \frac{(u(x)-u(y))^2}{|x-y|^N}\,dx\,dy = 0.
\end{equation}

From~\eqref{cnmrui342345ASDFGH0002} we deduce that~$u$ is a.e. constant, and therefore~$u$ must be identically zero, since we know that~$u=0$ in~$\D$.
But this fact contradicts~\eqref{cnmrui342345ASDFGH000}.
\end{proof}

Now we are ready to start our analysis of problem \eqref{eigenproblem}, and we aim to show that the functional $Q$ admits a global minimizer. In general we do not know if minimizing sequences are bounded, but the next result shows that it is always possible to modify the sequence appropriately in $\nN$ without increasing the energy of the functional in such a way that the sequence is bounded in the ``right'' space.

\begin{lemma} \label{lemma4}
Suppose that $\mathfrak{N}$ is bounded and that $\dist(\Omega,\mathfrak{N}) \geq d_0 > 0$.

Let $(u_k)_k \subset \HOO(\Omega \cup \nN)$ be a minimizing sequence for $Q$ such that
\[
\int_\Omega m(x) u_k^2 \, dx = 1.
\]
Using the notation from Lemma \ref{lemma1}, define
\begin{equation*}
v_k(x) :=
\begin{cases}
u_k(x), & \text{if } x \in \Omega \cup \D, \\[4pt]
\dfrac{E_{u_k}(x)}{E_1(x)}, & \text{if } x \in \nN .
\end{cases}
\end{equation*}

Then~$(v_k)_k$ is a minimizing sequence for $Q$
and is uniformly
bounded in $\HOO(\Omega \cup \nN)$, namely
$$
\sup_{k\in\N}\int_\Omega {v}_k^2\,dx+
 \iint_{ { \R^{2N} \setminus (\D \times \D)}\atop{\{|x-y| \leq 1\}}} \frac{(v_k(x) - v_k(y))^2}{|x-y|^N} 
\, dx \, dy <+ \infty.
$$
\end{lemma}

\begin{proof}
Thanks to Lemma \ref{lemma1}, the sequence $(v_k)_k$ is still minimizing for $Q$, and therefore
\begin{equation*}
  \sup_{k\in\N}  \iint_{\mathcal{Q} \cap\{ |x - y| \leq 1\}} 
\frac{(v_k(x) - v_k(y))^2}{|x - y|^N} \, dx \, dy<+\infty.
\end{equation*}
From this and
Proposition \ref{poincare}, we see that the sequence $(v_k)_k$ is
uniformly bounded in $L^2(\Omega)$. The thesis then 
follows from Proposition~\ref{boundedness}. 
\end{proof}

Now we are ready to show that the constrained minimum for $Q$ is attained.

\begin{proposition} \label{prop3}
Suppose that $\mathfrak{N}$ is bounded and that $\dist(\Omega,\mathfrak{N}) \geq d_0 > 0$.
 
Let $m \in L^\infty(\Omega)$.
Assume that there exists an open set $\mathcal{S} \subset \Omega$ such that~$m>0$ a.e. in~$\mathcal{S}$.

Then the infimum in the minimization problem
\begin{equation*}
    \lambda_1(m) := \inf \left\{ Q(u) \,\colon\, u \in \HOO(\Omega \cup \nN), \ \int_\Omega m(x) u^2(x) \, dx = 1 \right\}
\end{equation*}
is achieved by some non-negative function $e \in \HOO(\Omega \cup \nN)$, which is also a weak solution of problem \eqref{eigenproblem}.
\end{proposition}

\begin{proof}
Let~$\widetilde\varphi\in C^\infty_0(\mathcal{S},[0,1])$ 
and
$$ \varphi:=\frac{\widetilde\varphi}{\displaystyle\int_\Omega m(x) \widetilde\varphi^2(x) \, dx}.$$
We remark that~$\varphi\in\HOO(\Omega \cup \nN)$ and
$$ \int_\Omega m(x) \varphi^2(x) \, dx=1,$$
hence the set in which we perform the minimization is not empty.

Let now~$(u_k)_k \subset \HOO(\Omega \cup \nN)$ be a minimizing sequence. Thanks to the triangle inequality, we may assume without loss of generality that the sequence is non-negative. Moreover, if we define the  renormalized sequence 
    \begin{equation*}
        v_k(x) := 
        \begin{cases} 
            u_k(x), & \text{if } x \in \Omega \cup \D, \\ 
            \dfrac{E_{u_k}(x)}{E_1(x)}, & \text{if } x \in \nN,
        \end{cases}
    \end{equation*}
    we have, in view of Lemma \ref{lemma4},  that $(v_k)_k$ is still a minimizing sequence for $\lambda_1(m)$ and is uniformly
    bounded in $\HOO(\Omega \cup \nN)$. 
    
    Hence, by  the compact embedding $\HOO(\Omega \cup \nN) \hookrightarrow L^2(\Omega \cup \nN)$ (see \cite[Theorem 2.1]{MR3759570}), there exists $v \in \HOO(\Omega \cup \nN)$ such that, up to a subsequence, $v_k \to v$ strongly in $L^2(\Omega \cup \nN)$ as~$k\to+\infty$. In particular, $v_k \to v$ in $L^2(\Omega)$ as~$k\to+\infty$, and therefore
    \begin{equation*}
     \lim_{k\to+\infty}   \left|\int_\Omega m(x) \left( v_k^2 - v^2 \right) \, dx\right| \leq \| m \|_{L^\infty}\lim_{k\to+\infty} \int_\Omega |v_k^2 - v^2| \, dx = 0.
    \end{equation*}
 This implies that $v$ satisfies the constraint
    \begin{equation*}
        \int_\Omega m(x)  v^2(x) \, dx = 1.
    \end{equation*}
    Moreover, by the weak lower semi-continuity of the functional~$Q$, we deduce that~$v$ is a minimizer.

    Finally, it remains to prove that $v$ satisfies the Neumann boundary condition in $\nN$. In order to see this, we use the Lebesgue Dominated Convergence Theorem and
    we have that, for all~$x \in \nN$,
    \begin{equation*}
   v(x)=\lim_{k\to+\infty}    v_k(x)= \lim_{k\to+\infty}\frac{\displaystyle \int_{\substack{\Omega \\|x-z| \leq 1}} \frac{v_k(z)}{|x-z|^N}\, dz}{ \displaystyle \int_{\substack{\Omega \\|x-z| \leq 1}} \frac{1}{|x-z|^N} \, dz} \to \frac{\displaystyle \int_{\substack{\Omega \\|x-z| \leq 1}} \frac{v(z)}{|x-z|^N}\, dz}{ \displaystyle \int_{\substack{\Omega \\|x-z| \leq 1}} \frac{1}{|x-z|^N} \, dz}.
    \end{equation*}
    Thanks to Lemma \ref{lemma1}, this implies that $v$ satisfies the Neumann condition on $\nN$, and hence is a weak solution of problem \eqref{eigenproblem}. Setting $e:=v$ concludes the proof.
\end{proof}

We now deal with the minimization of the functional~$W$.

\begin{proposition} \label{prop4} 
Suppose that $\mathfrak{N}$ is bounded and that $\dist(\Omega,\mathfrak{N}) \geq d_0 > 0$.

Let $m \in L^\infty(\Omega)$. Then the infimum in the minimization problem
\begin{equation*}
    \widetilde{\lambda}_1(m) := \inf \left\{ W(u) \,\colon\, u \in \HOO(\Omega \cup \nN), \ \int_\Omega  u^2(x) \, dx = 1 \right\}
\end{equation*}
is achieved by some non-negative function $\widetilde{e} \in \HOO(\Omega \cup \nN)$, which is also a weak solution of problem \eqref{eigenschro}.
\end{proposition}

\begin{proof}
Let $(u_k)_k \subset \HOO(\Omega \cup \nN)$ be a minimizing sequence. without loss of generality, after renormalizing the sequence we may also assume that it is non-negative and satisfies~$\cN u_k = 0$ (thanks to Lemma~\ref{lemma1}).

Observe that
\begin{equation*}
    \left| \int_\Omega m(x) u_k^2 \right| \leq \| m \|_{L^\infty(\Omega)},
\end{equation*}
which entails that $$
\sup_{k\in\N} \frac12\iint_{ \cQ \cap\{ |x - y| \leq 1\}} \frac{(u_k(x)-u_k(y))^2}{|x-y|^N}\,dx\,dy
=\sup_{k\in\N} W(u_k)+\int_\Omega m(x) u_k^2 
<+\infty.$$
Thanks to this, we can apply Proposition \ref{boundedness}
and deduce that~$(u_k)_k$ is uniformly
bounded in~$\HOO(\Omega \cup \nN)$. 

Thus, there exists $\widetilde{e} \in \HOO(\Omega \cup \nN)$ such that $u_k \rightharpoonup \widetilde{e}$ as $k \to \infty$.  
Recalling the compact embedding $\HOO(\Omega \cup \nN) \hookrightarrow L^2(\Omega \cup \nN)$ (see \cite[Theorem 2.1]{MR3759570})
and the fact that $m \in L^\infty(\Omega)$, it follows that
\begin{equation*}
    \lim_{k \to \infty} \int_\Omega m(x) u_k^2 \, dx = \int_\Omega m(x) \widetilde{e}^2 \, dx
    \qquad\text{and}\qquad
     \int_\Omega \widetilde{e}^2 \, dx=\lim_{k \to \infty} \int_\Omega u_k^2 \, dx = 1.
\end{equation*}
Moreover, by the weak sequential lower semi-continuity of the norm, we have
\begin{equation*}
    \iint_{\cQ \cap\{ |x-y| \leq 1\}}
    \frac{(\widetilde{e}(x) - \widetilde{e}(y))^2}{|x - y|^N} \, dx \, dy\leq
    \liminf_{k \to+ \infty}
    \iint_{\cQ \cap \{|x-y| \leq 1\}}
    \frac{(u_k(x) - u_k(y))^2}{|x - y|^N} \, dx \, dy,
\end{equation*}
which shows that $\widetilde{e}$ is a minimizer, thus concluding the proof.
\end{proof}

\begin{proposition} \label{prop7}

Suppose that $\mathfrak{N}$ is bounded and that $\dist(\Omega,\mathfrak{N}) \geq d_0 > 0$.

Then there exists $\overline{m}^+>0$ such that, if $\Vert m^+ \Vert_{L^\infty(\Omega)} <\overline{m}^+, $ then $\widetilde{\lambda}_1(m) \geq 0$. 
\end{proposition}

\begin{proof}
    In view of Proposition \ref{poincare}, it is possible to find a constant $C>0$ such that
    \begin{equation*}
        W(u) \geq C \int_\Omega u^2 \, dx-\int_\Omega m(x) u^2 \, dx \geq \left( C - \Vert m^+ \Vert_{L^\infty(\Omega)}\right) \int_\Omega u^2 \, dx.
    \end{equation*}
Thus, if~$\Vert m^+ \Vert_{L^\infty(\Omega)}$ is sufficiently
small, then $W(u) \geq 0$ and the statement plainly follows.
\end{proof}

\section{Proof of Theorem \ref{th4}} \label{section5}

This section is devoted to the proof of Theorem \ref{th4}. The idea is to determine whether the minimizers of the functional are trivial or non-trivial. In this direction, a non-trivial minimum will exist if we can exhibit a suitable competitor and our candidate will be the first eigenfunction of problem \eqref{eigenproblem}, provided that $\lambda_1(m) < 1$. On the other hand, the minimum is trivial if we prove that $\widetilde{\lambda}_1(m) \geq 0$, since this implies that the functional is non-negative and therefore its minimizer is the zero function. The next result formalizes these ideas.

\begin{lemma} \label{lemma5}  
Suppose that $\nN$ is bounded and $\dist(\Omega, \nN) \geq d_0 > 0$. 

Let $m \in L^\infty(\Omega)$ and $\mu \in L^1(\Omega)$. 
Then:
\begin{itemize}
    \item if $\lambda_1(m) < 1$, then problem \eqref{mainlogistic} admits a non-trivial non-negative solution,
    \item if $\widetilde{\lambda}_1(m) \geq 0$ then the minimizers of $\cE$, which are also solutions of \eqref{mainlogistic}, are all trivial. 
\end{itemize}
\end{lemma}

\begin{proof}
By Theorem \ref{th1} we obtain the existence of a non-negative function $u$ minimizing the energy $\cE$ and solving the equation in \eqref{mainlogistic}.

We now prove that this solution is not identically zero under the condition $\lambda_1(m) < 1$. Let~$e$ be an eigenfunction corresponding to the first eigenvalue $\lambda_1:=\lambda_1(m)$ of \eqref{eigenproblem} as given by Proposition \ref{prop3}, with $e \in \HOO(\Omega \cup \nN)$ and
\begin{equation*} 
\frac{1}{2} \iint_{  \cQ \cap \{|x - y| \leq 1\}} \frac{(e(x) - e(y))(w(x) - w(y))}{|x - y|^N} \,dx\,dy = \lambda_1 \int_\Omega m e w\,dx,
\end{equation*}
for all $w \in \HOO(\Omega \cup \nN)$. Choosing $w := e$, we obtain
\begin{equation*}
\frac{1}{2} \iint_{ \cQ \cap \{|x - y| \leq 1\}} \frac{(e(x) - e(y))^2}{|x - y|^N} \,dx\,dy = \lambda_1 \int_\Omega m e^2\,dx,
\end{equation*}
and by renormalizing, we may assume
\begin{equation*}
\int_\Omega m e^2\, dx = 1.
\end{equation*}
Moreover, by a standard mollification argument and applying \cite[Proposition 4.1]{MR4002272} outside a compact region containing $\Omega \cup \nN$, it is not restrictive to suppose that~$e$ is bounded.

Let us evaluate \begin{align*}
\mathcal{E}(\varepsilon e) &= \frac{\varepsilon^2}{2} \left[ \frac{1}{2} \iint_{ \cQ \cap\{ |x - y| \leq 1\}} \frac{(e(x) - e(y))^2}{|x - y|^N} \,dx\,dy - \int_\Omega m e^2\, dx \right] + \frac{\varepsilon^3}{3} \int_\Omega \mu e^3\, dx \\
&= \frac{\varepsilon^2}{2} (\lambda_1 - 1) + \frac{\varepsilon^3}{3} \int_\Omega \mu e^3\, dx \\
&= -\frac{c_1}{2} \varepsilon^2 + \frac{c_2}{3} \varepsilon^3,
\end{align*}
where we define
\[
c_1 := 1 - \lambda_1 > 0 \qquad{\mbox{and}}\qquad c_2 := \int_\Omega \mu e^3\, dx.
\]
Since $e$ is bounded, $c_2$ is finite. For $\varepsilon$ sufficiently small, we thus
have that $\mathcal{E}(\varepsilon e) < 0 = \mathcal{E}(0)$ and thus the minimizer is not trivial.

Now, we turn our attention to the case $\widetilde{\lambda}_1:=\widetilde{\lambda}_1(m) \geq 0$. It follows from Proposition~\ref{prop4} that
\begin{equation*}
    \cE(u)\geq \widetilde{\lambda}_1 \int_\Omega u^2 \, dx + \frac{1}{3} \int_\Omega \mu(x)u^3 \, dx.
\end{equation*}
We observe that $\cE(u) \geq 0$ and $\cE(u)>0$ if $u \neq 0$ since $\mu(x) \geq \overline{\mu}$ and $u $ is non-negative. Thus  the only minimizer is the trivial one and this concludes the proof.   
\end{proof}

Now we look at conditions on the resource $m$ where the hypothesis $\lambda_1(m) <1$ is satisfied. In order do that, we study an optimization problem of $\lambda_1(m)$ with respect to $m$. We start by fixing some notation. Given $\overline{m}$,
$\underline{m}\in (0,+\infty)$ and $m_0\in (-\underline{m},0)$ 
we consider the class of resources 
\begin{equation}\label{resources}\begin{split}
\mathscr{M}=\mathscr{M}(\overline{m},\underline{m},m_0)
:=\; & \Big\lbrace m\in L^\infty(\Omega)\; {\mbox{ s.t. }}\;
\inf_\Omega m\ge -\underline{m},\quad
\sup_\Omega m\leq \overline{m},\\
&\qquad\qquad
\int_\Omega m(x)\,dx=m_0|\Omega|\quad{\mbox{ and }}\quad
m^+\not\equiv 0
\Big\rbrace.
\end{split}\end{equation}
We will also consider the smallest possible first eigenvalue
among all the resources in $\mathscr{M}$, namely we set
\begin{equation}\label{eq24}
\underline{\lambda}:=\inf_{m\in\mathscr{M}}\lambda_1(m).
\end{equation}

We first show that the distribution of resources that minimizes
the eigenvalue in \eqref{eq24} has a bang-bang structure, i.e.,
it only takes the extreme values $\underline{m}$ and $\overline{m}$.
This fact follows from the classical ``bathtub principle''
(see Lemma 3.3 in \cite{MR2660987}, or \cite{MR1415616, MR2281509}),
which we recall here for convenience:

\begin{lemma}\label{eq25}
Let $f\in L^1(\Omega)$ and ${\mathscr{M}}$ be as in \eqref{resources}.
Then the maximization problem 
\[
\sup_{m\in \mathscr{M}} \int_\Omega fm\,dx
\]
is attained by some $m\in\mathscr{M}$ of the form
\[
m:=\overline{m}\chi_D-\underline{m}\chi_{\Omega\setminus D},
\]
for a measurable subset $D\subset \Omega$
such that
\begin{equation}\label{eq22}
|D|=\frac{\underline{m}+m_0}{\underline{m}+\overline{m}}|\Omega|.
\end{equation}
\end{lemma}

In view of Lemma \ref{eq25},
the optimization of the eigenvalue $\lambda_1$ in \eqref{eq24}
can be restricted to bang-bang functions.
More precisely, we set
\begin{equation} \label{resources2} \begin{split}
\widetilde{\mathscr{M}}=
\widetilde{\mathscr{M}}(\overline{m},\underline{m},m_0):=\; &
\Bigg\lbrace m\in \mathscr{M} \;:\;
m=\overline{m}\chi_D-\underline{m}\chi_{\Omega\setminus D},\\ &\qquad
\text{for some set $D \subset \Omega$ with}\quad
|D|=\frac{\underline{m}+m_0}{\underline{m}+\overline{m}}|\Omega|
\Bigg\rbrace\end{split}
\end{equation}
and we obtain the following result:
 
\begin{proposition}\label{prop5}
We have
\begin{equation*}
\underline{\lambda}=\inf_{m\in\widetilde{\mathscr{M}}}\lambda_1(m).
\end{equation*}
\end{proposition}

\begin{proof}
Set
\[
\widetilde{\lambda}:=\inf_{m\in\widetilde{\mathscr{M}}}\lambda_1(m).
\]
We claim that 
\begin{equation}\label{eq26}
\underline{\lambda}=\widetilde{\lambda}.\end{equation}
Since $\widetilde{\mathscr{M}}\subset 
\mathscr{M}$, it is clear that
\begin{equation}\label{eq23}
\underline{\lambda}\leq\widetilde{\lambda}.\end{equation}
Conversely, by the definition of $\underline{\lambda}$ in \eqref{eq24},
for every $\varepsilon>0$ there exists $m_\varepsilon \in\mathscr{M}$
such that~$\underline{\lambda}+\varepsilon \geq \lambda_1(m_\varepsilon)$.  
Let $e_\varepsilon$ denote the
non-negative eigenfunction associated with~$\lambda_1(m_\varepsilon)$.
Then
\begin{equation}\label{eq27}
\underline{\lambda}+\varepsilon \geq
\lambda_1(m_\varepsilon)=
\frac{\displaystyle \frac12\iint_{ \cQ \cap\{ |x - y| \leq 1\}} \frac{ (e_\varepsilon(x)-e_\varepsilon(y))^2}{|x-y|^N}\,dx\,dy}{
\displaystyle\int_\Omega m_\epsilon e_\varepsilon ^2\,dx}.
\end{equation}
By Lemma \ref{eq25},
$$\int_\Omega m_\epsilon e_\varepsilon ^2\,dx\le 
\int_\Omega\big(\overline{m}
\chi_{D_\varepsilon}-\underline{m}\chi_{\Omega\setminus {D_\varepsilon}}\big)e_\varepsilon ^2\,dx,$$
for some ${D_\varepsilon}\subset\Omega$ satisfying~\eqref{eq22}.

Setting $m^\star_\epsilon:=
\overline{m}
\chi_{D_\varepsilon}-\underline{m}\chi_{\Omega\setminus {D_\varepsilon}}$,
we obtain
$$ \underline{\lambda}+\varepsilon \geq 
\frac{\displaystyle\frac12 \iint_{\cQ \cap\{ |x - y| \leq 1\}} \frac{ (e_\varepsilon(x)-e_\varepsilon(y))^2}{|x-y|^N}\,dx\,dy}{
\displaystyle\int_\Omega(\overline{m}
\chi_{D_\varepsilon}-\underline{m}\chi_{\Omega\setminus {D_\varepsilon}})e_\varepsilon ^2\,dx}
\geq \lambda_1 (m^\star_\epsilon )\ge \widetilde{\lambda}.
$$
Letting $\varepsilon\to0$, we conclude that $\underline{\lambda}\geq\widetilde{\lambda}$.
Together with \eqref{eq23},
this proves \eqref{eq26}.
\end{proof}

In lieu of Proposition \ref{prop5},
from now on, when optimizing $\lambda_1(m)$
as in \eqref{eq24}, we restrict ourselves to functions in the set $\widetilde{\mathscr{M}}$
defined in \eqref{resources2}. In this setting, our
next result shows that when $\overline{m}$ is very large, then $\underline{\lambda}(\overline{m},\underline{m},m_0)$ is very close to zero. 

\begin{proposition}\label{prop6} 
Assume that 
\begin{equation}\label{eq28}
\frac{\underline{m}+m_0}{\underline{m}+\overline{m}} \,\geq\, k_0
\end{equation}
for some constant $k_0>0$.  
Then there exists $C=C(\Omega,k_0)>0$ such that
\[
\underline{\lambda}(\overline{m},\underline{m},m_0)\;\leq\;\frac{C}{\overline{m}}.
\]
\end{proposition}

\begin{proof}
Let $B\subset \Omega$ be a ball such that 
\begin{equation}\label{eq29}
|B|\leq \frac{k_0}{2}|\Omega|.
\end{equation}
Up to translations, without loss of generality,
we may assume that~$\Omega\subset \lbrace x_n>0 \rbrace$.

For $\xi\ge0$, we define
\[
\Omega_\xi:=B \cup(\lbrace x_n<\xi \rbrace\cap \Omega).
\]
Clearly $|\Omega_\xi|$ is non-decreasing in $\xi$, and we set
\[
\xi^*:=\sup \left\lbrace \xi\ge0:\, 
|\Omega_\xi|<\frac{\underline{m}+m_0}{\underline{m}+\overline{m}}|\Omega| \right\rbrace.
\]

We claim that, for every $\underline{\xi}>0$,
\begin{equation}\label{eq30}
\lim_{\xi\to \underline{\xi}}|\Omega_{\xi}|=
|\Omega_{\underline{\xi}}|.\end{equation}
To see this, we show that
\begin{equation}\label{eq31}
\lim_{\xi\to \underline{\xi}}\chi_{\Omega_{\xi}}(x)=
\chi_{\Omega_{\underline{\xi}}}(x)\qquad {\mbox{ for a.e. }}x\in\Omega.
\end{equation}

If $x=(x',x_n)\in \Omega_{\underline{\xi}}$, then either $x\in B$ or $x_n<\underline{\xi}$.  
If $x\in B$, then $x\in \Omega_{\xi}$ for every $\xi>0$, and \eqref{eq31} holds since in this case we have $\chi_{\Omega_{\xi}}(x)=\chi_{\Omega_{\underline{\xi}}}(x)=1$.  

If instead $x_n<\underline{\xi}$, there exists 
$\widetilde{\xi}\in (x_n, \underline{\xi})$ such that, for all $\xi \in (\widetilde{\xi},\underline{\xi})$, we still have $\chi_{\Omega_{\xi}}(x)=\chi_{\Omega_{\underline{\xi}}}(x)=1$, which again gives \eqref{eq31}.  

On the other hand, if $x\not\in \Omega_{\underline{\xi}}$, then $x\not\in B$ and $x_n\geq \underline{\xi}$.  
Since the set $\{x_n=\underline{\xi}\}$ has measure zero, we may assume without loss of generality that $x_n> \underline{\xi}$.  
Then there exists $\widetilde{\xi}\in (\underline{\xi},x_n)$ such that, for all $\xi \in (\underline{\xi},\widetilde{\xi})$, we have $x\not\in\Omega_\xi$, which completes the proof of~\eqref{eq31}.  

Hence, by \eqref{eq31} and the Dominated Convergence Theorem, \eqref{eq30} follows.

Note also that if $\xi=0$, then $\Omega_\xi=B$, and therefore, by \eqref{eq28} and \eqref{eq29},
$$ |\Omega_\xi|=|B|\leq \frac{k_0}{2}|\Omega|\le
\frac{\underline{m}+m_0}{2\big(\underline{m}+\overline{m}\big)}
|\Omega|.
$$
Therefore, in view of the continuity result in~\eqref{eq30},
it follows that $\xi^*>0$.

Moreover, by \eqref{eq30}, we obtain
\begin{equation*}
|\Omega_{\xi^*}|=\frac{\underline{m}+m_0}{
\underline{m}+\overline{m}}\,|\Omega|.
\end{equation*}
As a consequence,
setting $D:=\Omega_{\xi^*}$, we see that $D$
satisfies \eqref{eq22}.
  
Now choose $v\in C^\infty_0(B)$ with $v\not\equiv 0$. Since $B\subset D$,
\[
\underline{\lambda}\leq 
\frac{\displaystyle \frac12\iint_{ \cQ \cap\{ |x - y| \leq 1\}} \frac{ (v(x)-v(y))^2}{|x-y|^N}\,dx\,dy}{\displaystyle
\int_\Omega\big(\overline{m}\chi_D-\underline{m}\chi_{\Omega\setminus D}\big)v^2\,dx}
=\frac{\displaystyle\frac12 \iint_{ \cQ \cap\{ |x - y| \leq 1\}} \frac{ (v(x)-v(y))^2}{|x-y|^N}\,dx\,dy}{\overline{m}\displaystyle\int_B v^2\,dx}
\le\frac{C}{\overline{m}},
\]
for some positive constant $C$ depending on $\Omega$ and $d_0$.
This completes the proof of Theorem~\ref{prop6}.
\end{proof}

\begin{proof}[Proof of Theorem \ref{th4}]
The existence of a global minimum for $\cE$ follows from Theorem~\ref{th1}. Moreover, in the first case this minimizer is non-trivial, as a consequence of Lemma~\ref{lemma5} and Proposition~\ref{prop6}. In the latter case, the minimum is trivial, again by Lemma~\ref{lemma5} and Proposition~\ref{prop7}.
\end{proof}

\section*{Acknowledgments}
This work has been supported by the
Australian Research Council Future Fellowship FT230100\-333
``New perspectives on nonlocal equations'' and
by the Australian Laureate Fellowship FL190100081 ``Minimal
surfaces, free boundaries and partial differential equations''.

The first author is supported by the EPSRC grant EP/W026597/1. He would also like to thank the Istituto Nazionale di Alta Matematica ``Francesco Severi'' (INdAM) for supporting a research visit to the University of Western Australia, where this project started, and the Mathematisches Forschungsinstitut Oberwolfach (MFO) for its hospitality during the Leibniz Fellowship program (OWLF), where part of this work was carried out.

\bibliographystyle{amsplain}
\bibliography{loglogistic}
\end{document}